\documentclass[12pt]{amsart} 

\usepackage[margin=3cm]{geometry}
\usepackage[english]{babel}
\usepackage{hyperref}
\usepackage{amsmath}
\usepackage{amsthm}
\usepackage{amssymb}
\usepackage{ucs}
\usepackage{enumerate}
\usepackage[utf8]{inputenc}
\usepackage[T1]{fontenc}
\usepackage{xcolor}
\usepackage{mathtools}
\usepackage{comment}
\usepackage{dsfont}

\usepackage[draft,inline,nomargin,index]{fixme}

\fxsetup{theme=color,mode=multiuser}
\FXRegisterAuthor{pf}{apf}{\color{olive}PF}
\FXRegisterAuthor{flm}{aflm}{\color{red}FLM}
\FXRegisterAuthor{km}{akm}{\color{purple}KM}
\FXRegisterAuthor{ip}{aip}{\color{blue}IP}

\DeclareMathOperator{\supp}{\mathrm{supp}}

\DeclareMathOperator{\Pfin}{\mathcal P_{\mathrm{fin}}}

\DeclareMathOperator{\Sub}{\mathcal S}
\DeclareMathOperator{\Subfin}{\mathcal S_{\mathrm{fin}}}

  \newcommand{\R}{\mathbb R}
  \newcommand{\N}{\mathbb N}
  \newcommand{\Q}{\mathbb Q}
  
  \newcommand{\C}{\mathbb C}

\DeclareMathOperator{\Real}{\mathrm{Re}}

  \newcommand{\ot}{\otimes}

  \newcommand{\inv}{^{-1}}
  \newcommand{\B}{\mathcal B}
  \renewcommand{\H}{\mathcal H}
  \newcommand{\BH}{\B(\H)}

\renewcommand{\leq}{\leqslant}
\renewcommand{\geq}{\geqslant}
\newcommand{\abs}[1]{\left\lvert #1\right\rvert}
\newcommand{\norm}[1]{\left\lVert #1\right\rVert}

\newcommand{\la}{\left\langle}
\newcommand{\ra}{\right\rangle}

\newcommand{\lowertopo}{\tau^{\mathrm{low}}}
\newcommand{\uppervietoristopo}{\tau^{\mathrm{up}}_{\mathrm{Viet}}}
\newcommand{\upperwijsmantopo}{\tau^{\mathrm{up}}_{\mathrm{Wijs}}}
\newcommand{\wijsmantopo}{\tau_{\mathrm{Wijs}}}
\newcommand{\vietoristopo}{\tau_{\mathrm{Viet}}}

\newcommand{\weaktopo}{\tau_{\mathrm{weak}}}
\newcommand{\strongtopo}{\tau_{\mathrm{strong}}}

\newcommand{\convexbalanced}{\mathcal F_{\mathrm{conv}}^{\mathrm{bal}}}
\newcommand{\closedconvexstar}{\mathcal F^*_{\mathrm{conv}}}

\newcommand{\h}{\mathcal H}
\newcommand{\bh}{\mathcal B(\h)}

\newtheorem{thmi}{Theorem}

\newtheorem{thm}{Theorem}[section]
\newtheorem{cor}[thm]{Corollary}
\newtheorem{lem}[thm]{Lemma}
\newtheorem{prop}[thm]{Proposition}
\newtheorem*{claim}{Claim}

\newenvironment{cproof}{\begin{proof}[Proof of the
        claim]}{\end{proof}}

\theoremstyle{definition}

\newtheorem{df}[thm]{Definition}
\newtheorem{rmk}[thm]{Remark}
\newtheorem*{ack}{Acknowledgments}

\newtheorem{countex}[thm]{Counterexample}

\newcommand{\SSig}{\mathbf\Sigma}
\newcommand{\PPi}{\mathbf\Pi}
 \newcommand{\Pizerothree}{\PPi^{0}_{3}}

\title[Michael's selection thm. and applications to the Maréchal topology]{Michael's selection theorem and applications to the Maréchal topology}

\author{Pierre Fima, François Le Maître, Kunal Mukherjee and Issan Patri}


\begin{document}
\maketitle
\begin{abstract}
The Maréchal topology, also called the Effros-Maréchal topology, is a natural topology one can put on the 
space of all von Neumann subalgebras of a given von Neumann algebra. 
It is a result of Maréchal from 1973 that this topology is Polish as soon as the ambient algebra has separable predual, but the sketch of proof in her research announcement 
appears to have a small gap. 
Our main goal in this paper is to fill this gap by a careful look at the topologies one can put on the space of weak-$*$ closed subspaces of a dual space.
We also indicate how Michael's selection theorem can be used 
as a step towards Maréchal's theorem, and how it simplifies the proof of an important selection result of Haagerup and Winsløw  for the Maréchal topology.
As an application,  we show that the space of finite von Neumann algebras is 
$\Pizerothree$-complete.
\end{abstract}


\section{Introduction}

In a 1973 research announcement, Odile Maréchal introduced a natural Polish topology on the space of von Neumann algebras acting on a given infinite dimensional separable Hilbert space \cite{marechalTopologieStructureBorelienne1973}, 
inducing the standard Borel space structure that Effros had defined thereon in 1965 \cite{effrosBorelspaceNeumann1965}. 
Although Maréchal's results already showed how this topology could lead 
to streamlined proofs for results on Borel subsets of the space of von Neumann algebras, her work faded from memory until Haagerup and Winsløw  published two beautiful papers in which this topology plays a central role.

While their first paper could be seen as a direct continuation of Maréchal's work, in the sense that they showed that many natural maps (including the commutant map) were continuous on the space of von Neumann algebras \cite{haagerupEffrosMarechalTopologySpace1998}, their second work revealed an unexpected connection with Connes' embedding conjecture \cite{haagerupEffrosMarechalTopology2000}. 
Indeed, they proved that the latter was equivalent to the density of the isomorphism class of the injective type III$_1$ factor in the space of all von Neumann algebras acting on a separable infinite dimensional Hilbert space.
This formulation of their result relies on the Baire category theorem, and hence on Maréchal's result that the space of von Neumann algebras on a separable Hilbert space is Polish. 
However, detailed proofs of Maréchal's theorem have never appeared. 
Our first goal in this paper is to provide such a proof. 
Here is a precise statement.

\begin{thmi}[Maréchal]\label{thmi: marechal}
Let $M$ be a von Neumann algebra with separable predual, denote by $\mathcal S(M)$ its space of von Neumann subalgebras (with the same unit as $M$), which we identify to their closed unit balls. 
Then $\mathcal S(M)$ is Polish for the Vietoris topology associated to the weak topology on the space of closed subsets of the unit ball of $(M)_1$. 
\end{thmi}

There is actually a small gap in the sketch of proof presented in \cite{marechalTopologieStructureBorelienne1973}: the space of unit balls of ultraweakly closed subspaces of $\BH$ is not compact (see Counterexample \ref{countex: not compact}).
We fill this gap by showing that the space of 
ultraweakly closed subspaces of $\BH$ is nevertheless Polish as soon as $\mathcal H$ is separable, which suffices to make Maréchal's proof go through (see Theorem \ref{thm:polish_vN}). 

We study more generally the space of weak-$*$ closed subspaces of 
a dual space $E^*$, and show the latter is always Polish as soon as $E$ is separable (see Corollary \ref{cor: weakstar grassmanian is Polish})
Our Theorem \ref{thm: duality} moreover shows that the well-known polar identification between
norm closed subspaces of $E$ and weak-$*$ closed subspaces of the dual is a homeomorphism, endowing the corresponding grassmanian spaces with the following topologies:
\begin{itemize}
    \item the space $\mathcal G_{\norm\cdot}(E)$ of norm closed subspaces of $E$ is endowed with the Wijsman topology associated to the norm;
    \item as in Theorem \ref{thmi: marechal}, the space $\mathcal G_{w*}(E^*)$ of weak-$*$ closed subspaces of $E^*$ is endowed with the Vietoris topology induced by the space of weak-$*$ closed subsets of its unit ball $(E^*)_1$, identifying a subspace to its closed unit ball.
\end{itemize}

Finally, as our title suggests, our paper relies on a version of Michael's selection theorem  that we prove in Section \ref{sec: michael}. 
This theorem allows one to continuously select elements from closed convex sets under a weak continuity assumption with respect to the lower topology (see Section \ref{sec: lower topo} for its definition). One also has to make some assumptions on the ambient normed space. We did not strive for generality and focused on having a short self-contained proof instead; a more general version of our Theorem \ref{thm: michael selection} can be found in \cite[Thm.~3.8.8]{sakaiGeometricAspectsGeneral2013}. Our density result (Corollary \ref{cor: michael dense selection}) appears not to have been explicitly stated in the literature but is of crucial importance towards our two applications.

The first application is the continuous selection of a dense subset of the unit ball of an ultraweakly closed subspace, a result due to Maréchal which she proved via a different approach (Proposition \ref{prop:weakly continous choice in the unit ball}). This selection result is a key step in her proof of the polishness of the Maréchal topology.
The second application is the Haagerup-Winsløw  selection theorem, which is the following very useful statement.

\begin{thmi}[Haagerup-Winsløw]\label{thmi: haagerup winslow}
    Let $M$ be a von Neumann algebra with separable predual, denote by $\mathcal S(M)$
    its space of subalgebras (with the same unit as $M$), endowed with the Maréchal topology.
    There is a sequence of Maréchal to strong-$*$ continuous maps $x_n:\mathcal S(M)\to M$ such that for every $N\in\mathcal S(M)$, $\{x_n(N)\colon n\in\N\}$ is strong-$*$ dense in $(N)_1$.
\end{thmi}

Just like Maréchal's proof of her selection result, the proof of Haagerup and Winsløw does require some Hilbert space ideas, endowing $\BH$ with a pre-Hilbert space structure. However, our proof does not. We simply identify the lower topologies associated to the weak topology and to the strong-$*$ topology (Proposition \ref{prop: same lower topologies on marechal}), and then we apply Michael's selection theorem.

As an application of Haagerup-Winsløw's selection theorem, we conclude our paper with a complexity calculation for the subset of finite von Neumann algebras acting on a separable Hilbert space.
Indeed, the presence of a topology on the space of von Neumann subalgebras of $\BH$ allows one 
to analyze in a much more precise manner Borel subsets thereof, via the so-called Borel hierarchy which roughly keeps track of how many countable unions/intersections one needs 
to carry out in order to obtain a specific subset, starting from closed/open subsets.
The following result says that the set of finite von Neumann algebras is a countable intersection of countable unions of closed subsets, and that this is sharp. 

\begin{thmi}\label{thmi: complexity finite}
    Let $\mathcal H$ be a separable infinite-dimensional Hilbert space.
    Denote by $\Subfin(\BH)\subseteq\Sub(\BH)$ the set of finite von Neumann subalgebras.
    Then $\Subfin(\BH)$ is $\Pizerothree$-complete for the Maréchal topology.
\end{thmi}

The fact that the set of finite von Neumann algebras is Borel is due to Nielsen \cite[Thm.~3.2]{nielsenBorelSetsNeumann1973}, but his proof relies on the separation theorem, and thus does not provide any additional complexity information. 
Our proof is much more elementary and relies on the fact that finite  von Neumann algebras are exactly those for which the adjoint map is uniformly continuous for the strong topology on bounded sets (see Theorem \ref{thm: chara finite}). 
Finally, this result should be compared to Haagerup and Winslow's Corollary 4.10 from \cite{haagerupEffrosMarechalTopologySpace1998}, which implies that the set of finite factors 
is $\mathbf{\Sigma}^0_2$-complete in the Polish space of factors.
The increase in complexity from factors to general von Neumann algebras reflects the fact that $\BH$ contains a copy of $\bigoplus_{n\in\N}\BH$, and that $\bigoplus_n M_n$ is finite iff for every $n\in\N$, the von Neumann algebra $M_n$ is finite (see the proof of Theorem \ref{thm: complexity finite}). 

Here is an outline of the paper. In Section \ref{sec:prelim}, we collect some preliminary
results on Polish spaces and on the topologies on $\BH$. In Section \ref{sec: topologies}, we review three topologies one can put on the space of closed subsets of a topological space: the lower topology, the Vietoris topology (in a compact Hausdorff space) and the Wijsman topology (in a metric space). Section \ref{sec: grassman} is devoted to the proof of our duality theorem, which identifies the Wijsman topology on closed subspaces of a normed space to the Vietoris topology on unit balls of weak-$*$ closed subspaces of the dual. In Section \ref{sec: michael}, we state and prove a version of Michael's theorem adapted to our needs. Finally Section \ref{sec: applications} is devoted to three applications, namely Maréchal's theorem (Theorem \ref{thmi: marechal}, see Theorem \ref{thm:polish_vN}), the Haagerup-Winsløw  selection theorem (Theorem \ref{thmi: haagerup winslow}, see Theorem \ref{thm:*strongly continous choice in the unit ball}), and finally Theorem \ref{thmi: complexity finite} (see Theorem \ref{thm: complexity finite}).

\begin{ack}
The present paper is an offspring of a long term project on the 
Maréchal topology. 
We have had the chance to carry out this 
project in several places dedicated to international collaboration, and 
we thus gratefully acknowledge funding from MFO,  ISI Delhi, SERB (SRG/2022/001717),  IRL ReLaX, IIT Madras (IoE Phase II, SB22231267MAETWO008573), ICISE, CEFIPRA (6101-1) and ANR (Project AODynG: ANR-19-CE40-0008 and Project ANR ANCG: ANR-19-CE40-0002).
\end{ack}

\section{Preliminaries}\label{sec:prelim}

\subsection{Polish spaces}

By definition a \textbf{Polish space} is a topological space which is separable and whose topology admits a compatible metric.
Since metrizable spaces are separable iff they are second-countable (in other words, their topology admits a countable basis), any closed subset of a Polish space is Polish for the induced topology.
Also note that a compact Hausdorff space is Polish iff it is second-countable, because Urysohn's metrization theorem provides  a compatible metric which has to be complete by compactness.

A subset of a Polish space $X$ is called $G_\delta$  if it can be written as a countable intersection of open subsets of $X$.
One of the most important facts regarding Polish spaces is that $G_\delta$ subsets of Polish spaces are exactly those which are Polish for the induced topology, see \cite[Thm.~3.1]{kechrisClassicaldescriptiveset1995} (in particular, closed subsets of Polish spaces are $G_\delta$, but this actually holds in any metrizable space).

Finally, a map $f:X\to Y$ between second-countable topological spaces $X$ and $Y$ is called \textbf{Baire class 1} if for every closed subset $F$ of $Y$, $f\inv(F)$ is a $G_\delta$ set, or equivalently if for every open subset $U$ of $Y$, $f\inv(U)$ is $F_\sigma$, namely it can be written as a countable union of closed subsets of $X$. Baire class 1 functions can be thought of being just below continuity in the "ladder" of regularity properties of functions. Indeed, key examples of Baire class 1 functions are lower semi-continuous and upper semi-continuous real valued functions. 

\subsection{The finite Borel hierarchy}

We now present an exposition of the first steps of the Borel hierarchy of sets, aimed towards readers who may not be familiar with (descriptive) set theory. The reader acquainted with this hierarchy may
skip the following two subsections after perhaps recalling Lemma \ref{lem: sequence of finite sets is Pi 0 3 complete}, which will be important in the sequel.

Given a Polish space $X$, the (finite) Borel hierarchy allows one to classify basic Borel subsets by the number of times one needs to take countable unions  and complements to obtain it from open subsets. This hierarchy is defined by induction on $n\geq 1$ by:
\begin{itemize}
    \item $\SSig^0_1(X)$ consists of all open subsets of $X$;
    \item $\PPi^0_1(X)$ consists of all closed subsets of $X$;
    \item for all $n\geq 1$, $\SSig^0_{n+1}(X)$ is the set of all
    countable unions of elements of $\PPi^0_n(X)$;
    \item for all $n\geq 1$, $\PPi^0_{n+1}(X)$ is the set of all countable intersections of elements of $\SSig^0_n(X)$.
\end{itemize}

Note that we have already encountered sets in the the basic levels of the hierarchy: $\PPi^0_2(X)$ consists precisely of the $G_\delta$ subsets of $X$, while $\SSig^0_2(X)$ consists of all $F_\sigma$ subsets of $X$. Further, this hierarchy also allows us to understand the precise way in which Baire class 1 functions are a generalisation of continuous functions - for continuous functions, the inverse image of $\PPi^0_1(X)$ subsets (viz. closed subsets of $X$) are also in $\PPi^0_1(X)$, while for Baire class 1 functions, inverse image of $\PPi^0_1(X)$ subsets of $X$ are in the larger class of $\PPi^0_2(X)$ sets. Indeed, the important fact that every closed subset is $G_\delta$ yields that $\PPi^0_1(X)\subseteq\PPi^0_2(X)$, which in turn yields by induction that for every $n\geq 1$, we have $\SSig^0_n(X)\subseteq \SSig^0_{n+1}(X)$ and 
$\PPi^0_n(X)\subseteq\PPi^0_{n+1}(X)$. 
Also observe that $A\subseteq X$ is $\SSig^0_n(X)$ if and only if its complement is in $\PPi^0_n(X)$. Finally, we note that subsets belonging to $\SSig^0_n(X)$ (resp. in $\PPi^0_n(X)$) are often simply called
$\SSig^0_n$ (resp. $\PPi^0_n$) subsets of $X$.


A crucial observation is that all these classes are stable under \emph{continuous reduction}, which is defined as follows: given Polish spaces $X$ and $Y$ and subsets $A\subseteq X$ and $B\subseteq Y$, we say that $A$ \textbf{continuously reduces} to $B$, and denote by $A\leq_c B$, if there exists a continuous function $f:X\to Y$ such that $A=f\inv(B)$. Then, it is easy to check that if $B\in \SSig^0_n(Y)$ (resp.~$B\in\PPi^0_n(Y)$) and we have $A\subseteq X$ such that $A\leq_c B$, then $A\in \SSig^0_n(X)$ (resp.~$A\in\PPi^0_n(X)$).

We now define completeness, allowing us to make sense of subsets inside a given complexity class which are as complicated as possible. 

Let us recall that a Polish space is \textbf{zero-dimensional} when its topology
admits a basis consisting of clopen subsets.
The Baire space $\N^\N$ is arguably the
most important example of zero-dimensional Polish space, since it surjects continuously onto any Polish space (see \cite[Thm.~7.9]{kechrisClassicaldescriptiveset1995}).
 In particular, given any subset of a Polish space $Y$, there is a subset of $\N^\N$ which continuously reduces to it.
This motivates the following definition.

\begin{df}
Let $Y$ be a Polish space. Then, for any $n\geq 1$, a subset $B\subseteq Y$ is called $\SSig^0_n$\textbf{-hard} (resp.~$\PPi^0_n$\textbf{-hard}) if whenever $X$ is a zero-dimensional Polish space and $A\in\SSig^0_n(X)$ (resp.~$A\in\PPi^0_n(X)$), we have $A\leq_c B$. 

Finally, $B\subseteq Y$ is $\SSig^0_n$-complete if it is both $\SSig^0_n$ and $\SSig^0_n$-hard, and $\PPi^0_n$-complete
if it is both $\PPi^0_n$ and $\PPi^0_n$-hard.
\end{df}

\begin{rmk}
Since every zero-dimensional Polish space is homeomorphic to a closed subspace of $\N^\N$ (see \cite[Thm~7.8]{kechrisClassicaldescriptiveset1995}), one could equivalently always take $X=\N^\N$ in the above definition.
\end{rmk}
It follows from \cite[Thm.~22.4]{kechrisClassicaldescriptiveset1995} that if $B\subseteq Y$ is  $\SSig^0_n$-complete then it is $\SSig^0_n$ but not $\PPi^0_n$, and dually that if $B\subseteq Y$ is $\PPi^0_n$-complete then it is $\PPi^0_n$ but not $\SSig^0_n$. 
\begin{rmk}
Using \emph{Borel determinacy},
Wadge proved
that the converse is true: if $B\subseteq Y$ is $\SSig^0_n$ but not $\PPi^0_n$ then it is $\SSig^0_n$-complete; dually if $B\subseteq Y$ is $\PPi^0_n$ but not $\SSig^0_n$ then it is $\PPi^0_n$-complete, see \cite[Thm.~22.10]{kechrisClassicaldescriptiveset1995}.
\end{rmk}

It follows from the transitivity of $\leq_c$ that if $A$ is $\SSig^0_n$-hard, and $A\leq_c B$ then $B$ is $\SSig^0_n$-hard, and taking complements that 
if $A$ is $\PPi^0_n$-hard and $A\leq_c B$ then $B$ is $\PPi^0_n$-hard.

\begin{rmk}
The finite Borel hierarchy which we present here does not exhaust all Borel sets. 
The interested reader unfamiliar with ordinals
should consult \cite{srivastavacourseBorelsets1998} for a self-contained account of the full Borel hierarchy.
\end{rmk}

\subsection{Climbing up the Borel hierarchy}
We now give a short self-contained account of two
basic examples of benchmark sets in the Borel hierarchy.

\begin{lem}\label{lemma:finite sets are sigma 0 2 complete}
    $\Pfin(\N)$ is a $\SSig^0_2$-complete subset of $\{0,1\}^\N$.
\end{lem}
\begin{proof}
    Since $\Pfin(\N)$ is countable, it is $F_\sigma$, i.e.~$\SSig^0_2$.
    Since there is a bijection $\N\times\N\to\N$, we will show $\Pfin(\N)$ is $\SSig^0_2$-hard by  equivalently showing that the set $\Pfin(\N\times\N)$
    of finite subsets of $\N\times\N$ is a $\SSig^0_2$-hard subset of $\{0,1\}^{\N\times\N}$.
    
    So let $X$ be zero-dimensional Polish, suppose $A\subseteq X$ is 
    $\SSig^0_2$, and write $A=\bigcup_{n\in\N} A_n$ where
    each $A_n$ is closed. Replacing $A_n$
    by $\bigcup_{i\leq n} A_i$, we assume without loss of generality
    that $A_n\subseteq A_{n+1}$. 
    
    For every $n\in\N$, since $X$ is zero-dimensional, we can write
    \[
    X\setminus A_n=\bigcup_{m} U_{n,m}
    \]
    where each $U_{n,m}$ is a clopen subset of $X$.
    Furthermore, by replacing each $U_{n,m}$ by $U_{n,m}\setminus\bigcup_{j<m}U_{n,j}$, we assume without loss
    of generality that for a fixed $n\in\N$, the family $(U_{n,m})_m$
    consists of disjoint clopen subsets of $X$. 
    
    Now consider the map $f:X\to \{0,1\}^{\N\times\N}$, $x\mapsto (1_{U_{n,m}}(x))_{(n,m)\in\N\times\N}$, where $1_{U_{n,m}}$ denotes the characteristic function of $U_{n,m}$. Since each $U_{n,m}$ is clopen, the map $f$ is continuous. 
    
    In order to conclude the proof, we have to show that $A=f\inv(\Pfin(\N\times\N))$, i.e.~that for all $x\in X$, we have 
    $x\in A$ iff there are finitely many $(n,m)\in\N\times\N$ such that $x\in U_{n,m}$.
    By construction, for every $n\in\N$, there is at most one $m\in\N$ such that $x\in U_{n,m}$. 
    Now if $x\in A$, let $n_0\in\N$ such that $x\in A_{n_0}$, then for all $n\geq n_0$ we have $x\in A_n$ and hence $x\notin U_{n,m}$ for all $m\in\N$. We conclude that $x$ belongs to finitely many $U_{n,m}$. 
    Conversely if $x\not\in A$ then for every $n\in\N$ there is $m\in\N$ such that $x\in U_{n,m}$, and so there are infinitely many $(n,m)\in\N\times\N$ such that $x\in U_{n,m}$.
    This shows the desired equivalence, which finishes the proof.
\end{proof}

We now present a standard tool for climbing up the Borel hierarchy.

\begin{prop}\label{prop: build pi complete from sig complete}Let $Y$ be a Polish space, suppose $B\subseteq Y$ is $\SSig^0_n$-complete. Then $B^\N\subseteq Y^\N$ is $\PPi^0_{n+1}$-complete. 
\end{prop}
\begin{proof}
First note that $B^\N=\bigcap_{k\in\N}\pi_k\inv(B)$ where $\pi_k$ is the projection on the $k$-th coordinate, so that $B^\N$ is in $\PPi^0_{n+1}(Y^\N)$. 

Now suppose $A\in\PPi^0_{n+1}(\N^\N)$. By definition there is a countable family $(A_k)$ of elements of $\SSig^0_n(\N^\N)$ such that $A=\bigcap_{k\in\N}A_k$. 
Consider the map 
$$ \left.\begin{array}{rl}\Phi:\N^\N & \to(\N^\N)^\N \\x &\mapsto(x,x,...)\end{array}\right..$$
Then $\Phi$ is continuous and $A=\Phi\inv(\prod_{k\in\N} A_k)$. We thus only need to find $\Psi: (\N^\N)^\N\to Y^\N$ continuous such that $\prod_{k\in\N} A_k=\Psi\inv(B^\N)$. 
But since $B$ is $\SSig^0_n$ complete we have for every $k\in\N$ a continuous map $f_k: \N^\N\to Y$ such that $A_k=f_k\inv(B)$. The map $\Psi\coloneqq (f_k)_{k\in\N}:(\N^\N)^\N\to Y^\N$  satisfies $\prod_{k\in\N} A_k=\Psi\inv(B^\N)$, so that 
$A=(\Psi\circ\Phi)\inv(B^\N)$ as wanted.
\end{proof}

Our final lemma follows directly from the previous proposition along with Lemma \ref{lemma:finite sets are sigma 0 2 complete}.

\begin{lem}\label{lem: sequence of finite sets is Pi 0 3 complete}
The set $\Pfin(\N)^\N$ is $\PPi^0_3$-complete in $(\{0,1\}^\N)^\N$.\qed
\end{lem}

In other words, the previous lemma shows that the set of sequences of finite subsets of $\N$ is $\PPi^0_3$-complete. This will be crucial for us in the sequel.

\subsection{Topologies on \texorpdfstring{$\BH$}{B(H)}}

Let $\mathcal H$ be a Hilbert space. We will use the following three topologies on the algebra $\BH$ of bounded operators on $\mathcal H$, which we describe in terms of nets: 
\begin{itemize}
    \item $x_i\to x$ \textbf{weakly} if for every $\xi,\eta\in\mathcal H$ we have 
    $\la x_i\xi,\eta\ra\to \la x\xi,\eta\ra$.
    \item $x_i\to x$ \textbf{strongly} if for every $\xi\in\mathcal H$ we have $\norm{x_i\xi- x\xi}\to 0$.
    \item $x_i\to x$ \textbf{$*$-strongly} if for every $\xi\in\mathcal H$ we have both $\norm{x_i\xi- x\xi}\to 0$ and 
    $\norm{x_i^*\xi-x^*\xi}\to 0$.
\end{itemize}
Since $\BH$ is the dual of the Banach space of trace class operators 
(see for instance \cite[Thm.~3.4.13]{pedersenAnalysisNow1989}), we can also endow it with the associated weak-$*$ topology, which is called the \textbf{ultraweak topology} (or $\sigma$-weak topology). The latter refines the weak topology and coincides with it on the closed unit ball $\BH_1$ (see \cite[Prop.~4.6.14]{pedersenAnalysisNow1989}). In particular $\BH_1$ is weakly/ultraweakly compact by the Banach-Alaoglu theorem and Polish
as soon as $\mathcal H$ is separable.

We also need the following basic separate continuity result for the weak topology towards proving Maréchal's theorem. 
\begin{lem}\label{lem: separate continuity mult weak}
Let $x\in\BH$. The maps $y\mapsto yx$ and $y\mapsto xy$ are continuous for the weak topology. 
\end{lem}
\begin{proof}
    The continuity of $y\mapsto yx$ for the weak topology is immediate to check. 
    For $y\mapsto xy$, one reduces it to the previous case using the adjoint map, which is a homeomorphism for the weak topology and satisfies $(xy)^*=y^*x^*$.
\end{proof}
We arrive at the following important observation, already present in Maréchal's paper.
\begin{prop}\label{prop: mult is weakly baire class 1}
	Let $\mathcal{H}$ be separable. The multiplication map $(x,y)\mapsto xy$ is a Baire class 1 map $(\BH_1,\weaktopo)\times(\BH_1,\weaktopo)\to(\BH_1,\weaktopo)$.
\end{prop}
\begin{proof}
	Since $\BH_1$ is compact Polish for the weak topology, and since the above lemma precisely says that multiplication is separately continuous on $\BH_1$ for the weak topology, the proposition follows from 
	\cite[§31.V, Thm.~2]{kuratowskiTopologyVolume1966}.
\end{proof}

We require one last lemma on $\BH$ for our proof of the Haagerup-Winsløw  selection theorem. This lemma is stated as Lemma 2.4 in \cite{haagerupEffrosMarechalTopologySpace1998}. It can be seen as a generalization of the well-known fact that on the unitary group, the weak and the strong-$*$ topology are the same, and the proof is indeed very similar. We provide the proof for the reader's convenience.

\begin{lem}\label{lem: unitary have same weak and strong nbhd on B1}
	Let $u\in\B(\h)$ be an isometry, then for every strong neighborhood $U$ of $u$, there is a weak neighborhood $V$ of $u$ such that $V\cap \BH_1\subseteq U\cap \BH_1$. If $u$ is moreover unitary, then the same conclusion holds for any strong* neighborhood $U$ of $u$.
\end{lem}
\begin{proof}
	It suffices to prove this when $U$ is a subbasic neighborhood of $u$, so that we can then take $U$ to be of the form $U=\{x: \norm{(x-u)\xi}^2<\epsilon\}$ for the strong topology, while for the strong-$*$ topology it could also be of the form  $U=\{x:\norm{(x^*-u^*)\xi}^2<\epsilon\}$,
	where $\xi$ is a vector of norm $1$ and $\epsilon>0$.
	Let us start by the case $U=\{x: \norm{(x-u)\xi}^2<\epsilon\}$ and $u$ is an isometry. Note that for any $x\in\BH_1$, we have 
	\begin{align*}
		\norm{(x-u)\xi}^2&=\norm{u\xi}^2+\norm{x\xi}^2-2 \Real\la x\xi,u\xi\ra\\
		&\leq 2-2 \Real\la x\xi,u\xi\ra\\
		&=2 \Real\la (u-x)\xi,u\xi\ra,\text{ since }u\text{ is isometric.}
	\end{align*}
	So if we let $V=\{x\in\BH\colon \abs{\la (x-u)\xi,u\xi\ra}<\epsilon/2 \}$, then the previous calculation yields $V\cap \mathcal B(\mathcal H)_1\subseteq U\cap \mathcal B(\mathcal H)_1$ as wanted.
	
	For the strong-$*$ topology, suppose moreover that $u$ is unitary. Then we need to deal with the additional case where
	\[
	U=\{x:\norm{(x^*-u^*)\xi}^2<\epsilon\}.
	\]
	However, the above calculation applied to the isometry $u^*$  yields 
	\[
	\norm{(x^*-u^*)\xi}^2\leq 2 \Real\la (u^*-x^*)\xi,u^*\xi\ra=2 \Real \la (u-x)u^*\xi,\xi\ra.
	\]
	So if we let this time $V=\{x\in\BH\colon \abs{\la (u-x)u^*\xi,\xi\ra}<\epsilon/2 \}$, then $V\cap \mathcal B(\mathcal H)_1\subseteq U\cap \mathcal B(\mathcal H)_1$ as wanted.
\end{proof}

\section{Topologies on the space of closed subsets}\label{sec: topologies}

Given a topological space $X$, we denote by $\mathcal F(X)$
the space of all closed subsets of $X$. 
Several relevant topologies can be put on $\mathcal F(X)$, see e.g.~\cite{beerTopologiesClosedClosed1993}, and 
our work requires three of them. 
We start with the lower topology, which will be
refined by the two other topologies we describe, namely the Wijsman topology
and the Vietoris topology. 

\subsection{The lower topology}\label{sec: lower topo}

Given a topological space $X$, the \textbf{lower topology} (also called lower Vietoris topology or lower semi-finite topology) on 
$\mathcal F(X)$ is obtained by declaring the following sets to be open: 
\[
\mathcal I_U\coloneqq \{F\in\mathcal F(X)\colon F\cap U\neq \emptyset\},
\]
where $U$ is an open subset of $X$.
Sets of the form $\mathcal I_U$ where $U$ is open thus form a subbasis for the lower topology, and we denote the later by $\lowertopo$. 
Moreover, if $(U_i)_{i\in I}$ is a basis of the topology of $X$, it is straightforward to 
show that $(\mathcal V_{U_i})_{i\in I}$ is a subbasis for the lower topology on $\mathcal F(X)$. 
In particular, if $X$ is second-countable, then so is the lower topology on $\mathcal F(X)$.

Also observe that if $X$ and $Y$ are topological spaces, the map 
\[
\begin{array}{ccc}
    \mathcal F(X)\times\mathcal F(Y)&\to& \mathcal F(X\times Y)\\
     (A,B)&\mapsto&A\times B 
\end{array}
\]
is an embedding of topological spaces for the lower topology, allowing us to  view 
$\mathcal F(X)\times\mathcal F(Y)$ as a topological subspace of $\mathcal F(X\times Y)$.

The lower topology satisfies the following important continuity property.

\begin{lem} \label{lem: induced continuity for lower}
    Let $X$ and $Y$ be topological spaces, let $f:X\to Y$ be continuous. Then the map $f_*:\mathcal F(X)\to\mathcal F(Y)$ defined by $f_*(F)=\overline{f(F)}$ is continuous, 
    if we endow $\mathcal F(X)$ and $\mathcal F(Y)$ with their respective lower topologies. 
\end{lem}
\begin{proof}
    Let $U$ be an open subset of $Y$. Then by definition $f_*\inv(\mathcal I_U)$ is the set of all $F\in\mathcal F(X)$ such that $\overline{f(F)}\cap U\neq\emptyset$, which is equivalent to $f(F)\cap U\neq \emptyset$. The latter condition is finally equivalent to $F\cap f\inv(U)\neq\emptyset$, which means that $f_*\inv(\mathcal I_U)=\mathcal I_{f\inv(U)}$. Since $f$ is continuous, this shows that the $f_*$-preimage of any subbasic open set of the form $\mathcal I_U$ is open, so $f_*$ is continuous as wanted.
\end{proof}

Note that the map $X\rightarrow \mathcal{F}(X)$, $x\mapsto\{x\}$ is continuous for the lower topology. Also, the intersection map $(F_1,F_2)\mapsto F_1\cap F_2$ is not continuous for the lower topology in general, for instance if $x_n\to x$ with $x_n\neq x$ then 
$\{x_n\}\cap\{x\}=\emptyset\not\to \{x\}$ for the lower topology.
However, the intersection map satisfies the following weaker continuity property which will be useful in establishing the Michael's Selection Theorem (Theorem \ref{thm: michael selection}).

\begin{lem}\label{lem: continuity of intersection with fattenings}
	Let $W\subseteq X\times X$ be open, where $X$ is any topological space and define $W_x\coloneqq\{y\in X\,:\,(x,y)\in W\}$. The
	map 
	$$
	X\times\mathcal F(X)\rightarrow \mathcal F(X),\quad
	(x,F)\mapsto \overline{W_x\cap F}
	$$
	is continuous for the lower topology.
\end{lem}
\begin{proof}
	Let $(x_0,F_0)\in X\times \mathcal F(X)$ and suppose that $U$ is an open subset of $X$ such
	that $\overline{W_{x_0}\cap F_0}\cap U\neq\emptyset$. we need to find an open neighborhood $O$
	of $(x_0,F_0)$ such that we still have $\overline{W_{x}\cap F}\cap U\neq\emptyset$ for all $(x,F)\in O$.
	
	We have that $W_{x_0}\cap F_0\cap U$ is not empty, so we find $y_0\in U\cap F_0$ such that $(x_0,y_0)\in W$. 
	Since $W$ is open, we find open neighborhoods $V_1$ of $x_0$ and $V_2\subseteq U$ of $y_0$ such that $V_1\times V_2\subseteq W$.
	It follows that for any $x\in V_1$, we have $V_2\subseteq W_{x}$, so for any $F$ such that $F\cap V_2\neq \emptyset$,
	we have $F\cap W_x\cap U\neq \emptyset$ since $V_2$ is contained in $U$. It follows that \[
	O= V_1\times\{F\in\mathcal F(X)\colon F\cap V_2\neq \emptyset \}
	\]
	is the desired neighborhood of $(x_0,F_0)$.
\end{proof}

\subsection{The Vietoris topology}\label{sec: vietoris topo}

Let $X$ be a compact Hausdorff space. Then $\mathcal F(X)$ is equal to the space of compact subsets of $X$, and we endow it with the \textbf{Vietoris topology} which is defined as the join of the lower topology 
and the \emph{upper Vietoris topology}. The latter is the topology whose basis is given by sets of the form 
$$\mathcal C_U\coloneqq\{ F\in\mathcal F(X)\colon F\subseteq U\},$$ 
where $U$ is any open subset of $X$. We denote by $\uppervietoristopo$ the upper Vietoris topology, and by $\vietoristopo$ the Vietoris topology. Note that the empty set is then isolated for the upper Vietoris topology, via the open set $\mathcal C_{\emptyset}$. Also note that if $X'\subseteq X$ is closed, then the topology induced by the Vietoris topology of $\mathcal F(X)$ on $\mathcal F(X')$ is equal to the Vietoris topology of $\mathcal F(X')$.

The following fundamental theorem is well-known, see e.g.~\cite[4.9.6]{michaelTopologiesSpacesSubsets1951} where the Vietoris topology is called the \emph{finite topology}. 

\begin{thm}[Vietoris]\label{thm: Vietoris is compact Hausdorff}
    Let $X$ be a compact Hausdorff space. Then the Vietoris topology on $\mathcal F(X)$ is compact and  Hausdorff.
\end{thm}

The next lemma will be useful towards proving some natural subspaces of $\mathcal F(X)$ are closed, hence compact by the previous theorem.

\begin{lem} \label{lem: closed inclusion vietoris}
    If $X$ is a compact Hausdorff space then the set
    $$\{(F,G)\in\mathcal F(X)\times \mathcal{F}(X)\text{ such that }
    F\subseteq G\}$$
    is closed in $\mathcal F(X)\times \mathcal F(X)$ equipped
    with the product of the \emph{lower} topology with the \emph{upper} Vietoris topology, namely $\lowertopo\times\uppervietoristopo$. In particular it is closed for $\lowertopo\times\vietoristopo$ and for $\vietoristopo\times\vietoristopo$.
\end{lem}
\begin{proof}
    We show that the complement is open: take $(F_0,G_0)$ such that $F_0\not\subseteq G_0$. Fix some $x_0\in F_0\setminus G_0$, then since $X$
    is compact we find $U$ open containing $x_0$ disjoint from $V$ open containing $G_0$. 
    The set $\mathcal I_U\times \mathcal C_V$ is an open neighborhood of $(F_0,G_0)$ all whose elements $(F,G)$ satisfy $F\not\subseteq G$ as wanted.
\end{proof}

Specifying to singletons, we have the following fact.
\begin{lem}\label{lem: belonging is closed}
Let $(X,\tau_X)$ be a compact Hausdorff space. 
Then the set of couples $(x,F)$ such that $x\in F$ is closed for the product topology $\tau_X\times\uppervietoristopo$, in particular it is closed for $\tau_X\times\vietoristopo$.
\end{lem}
\begin{proof}
    The map $x\in X\mapsto\{x\}\in\mathcal F(X)$ is $\tau_X$ to $\lowertopo$ continuous (actually, it is a topological embedding), so the result follows immediately from the previous lemma.
\end{proof}

Finally, if $X$ is compact Polish, its topology admits a countable basis $(U_n)_{n\in\N}$, and it can easily be checked using compactness that the following sets form a basis for the upper Vietoris topology:
$$
\mathcal C_{\bigcup_{n\in F} U_n}, \text{where } F \text{ ranges over finite subsets of }\N.
$$
Since the lower topology is also secound-countable, 
the Vietoris topology is second-countable as well, 
hence Polish (see also \cite[Sec.~4.7]{kechrisClassicaldescriptiveset1995} for a more general result on the space of compact subsets of a Polish space).

\subsection{The Wijsman topology}\label{sec: wijsman}

Denote by $\mathcal F^*(X)$ the space of closed non-empty subsets of a topological space $X$.
In order to motivate the definition of the Wijsman topology, we begin by noting the following nice way of understanding the lower topology, when $X$ is equipped with a compatible metric.

\begin{prop}\label{prop: hyperspace lsc is usc distance function}
Suppose $(X,d)$ is a metric space. 
Then the lower topology induced on the set $\mathcal 
F^*(X)$ is the coarsest topology which makes the map $F\mapsto 
d(x,F)$ upper semi-continuous for every $x\in X$.
\end{prop}
\begin{proof}
	Recall that a function $f:X\rightarrow \mathbb{R}$ is upper semi-continuous if for any $r\in \mathbb{R}$, the set $\{x\in X: f(x)<y\}$ is open in $X$. The result now follows from the fact that a subbasis for the lower topology is given by sets of the form $\{F\in\mathcal F^*(X): B(x,r)\cap F\neq\emptyset\}$, which can be rewritten as $\{F\in\mathcal F^*(X): d(x,F)<r\}$.
\end{proof}

For a metric space $(X,d)$, we now endow its space $\mathcal F^*(X)$ with the \textbf{Wijsman topology}, which is the coarsest topology such that for all $x\in X$, the map $F\in\mathcal F^*(X)\mapsto d(x,F)$ is continuous. We denote the Wijsman topology by $\wijsmantopo$.
By the previous lemma, the Wijsman topology refines the lower topology, and it could be defined as the join of the lower topology and the \emph{upper Wijsman topology} $\upperwijsmantopo$, which is defined as the coarsest topology making the map $F\mapsto d(x,F)$ lower semi-continuous, for all $x\in X$. We then have the following analogue of Lemma \ref{lem: closed inclusion vietoris}.

\begin{lem}\label{lem: containement is closed for wijsman}
    Let $(X,d)$ be a metric space. Then the set of all couples $(F,G)\in\mathcal F^*(X)\times \mathcal F^*(X)$ such that $F\subseteq G$ is closed in the
    $\lowertopo\times\upperwijsmantopo$ topology, whence it follows that it is also closed in
    $\lowertopo\times\wijsmantopo$ and $\wijsmantopo\times\wijsmantopo$ topologies.
\end{lem}
\begin{proof}
    Suppose $F_0\not\subseteq G_0$, take $x_0\in F_0\setminus G_0$, then $d(x_0,G_0)>0$. Let $\epsilon=d(x_0,G_0)/3$. By the previous lemma, the set of all $F$ such that $d(x_0,F)<\epsilon$ is $\lowertopo$-open, while the set of all $G$ such that $d(x_0,G_0)>2\epsilon$ is $\upperwijsmantopo$ open by definition. 
    Thus, the set of $(F,G)$ satisfying $d(x_0, F)<\epsilon$ and $d(x_0,G)>2\epsilon$ defines a $\lowertopo\times\upperwijsmantopo$ open neighbourhood of $(F_0,G_0)$, for which we clearly have $F\not\subseteq G$ for every such $(F,G)$, which finishes the proof. 
\end{proof}

The following important theorem is due to Beer.

\begin{thm}[{\cite[Thm.~4.3]{beerPolishtopologyclosed1991}}]
\label{thm: Wijsman Polish}
    Let $(X,d)$ be a complete separable metric space. 
    Then $\mathcal F^*(X)$ is Polish for the Wijsman topology
    associated to the metric $d$.
\end{thm}

\begin{rmk}
For $(X,d)$ a compact metric space, the Vietoris topology conincides with the Wijsman topology on $\mathcal F^*(X)$. Indeed, it is not hard to see that the Wijsman topology is refined by the Vietoris topology (as for instance shown in the beginning of the proof of \cite[Thm.~2.2.5]{beerTopologiesClosedClosed1993}),
so since the latter is compact and Hausdorff (Theorem \ref{thm: Vietoris is compact Hausdorff}), they have to coincide. 
\end{rmk}

\section{Duality for grassmanians}\label{sec: grassman}

\subsection{Vietoris topology on closed subsets of the unit ball of the dual}

Let $(E,\norm\cdot)$ be a normed vector space, then consider its dual $E^*$, which is a Banach space for the norm $\norm\omega\coloneqq\sup_{x\in (E)_1} \abs{\omega(x)}$, where $(E)_1$ is the closed unit ball of $E$. Further, the dual $E^*$ can be equipped
with the weak-$*$ topology, defined as the weakest topology making the map
$\omega\in E^*\mapsto \omega(x)$ continuous for every $x\in E$. 
The closed unit ball $(E^*)_1$ is then weak-$*$ compact as a consequence of the Banach-Alaoglu theorem. We can thus equip its space of closed (equivalently compact) subsets $\mathcal F((E^*)_1)$ with the Vietoris topology as defined in Section \ref{sec: vietoris topo}.

For every $F\in \mathcal F((E^*)_1)$ non-empty and $x\in E$, define 
\begin{equation}\label{eq: definition of rhoF}
   \rho_F(x)=\sup_{\omega\in F}\abs{\omega(x)}=\max_{\omega\in F}\abs{\omega(x)}
\end{equation}
and by convention let $\rho_\emptyset(x)=0$. Note that by the Hahn-Banach theorem,
if $F=(E^*)_1$, then $\rho_F=\norm{\cdot}$.

\begin{lem}\label{lem: pointwise continuity}
    Fix $x\in E$. Then the map 
    \[
    F\in \mathcal F((E^*)_1)\mapsto \rho_F(x)
    \]
    is continuous for the Vietoris topology.
\end{lem}
\begin{proof}
        Since the empty set is isolated for the Vietoris topology, 
    we may as well take $F_0\in\mathcal F((E^*)_1)$ non empty and $\epsilon>0$; we then need to find a neighborhood $\mathcal U$ of $F_0$ such that for all 
    $F\in \mathcal U$ we have $\abs{\rho_F(x)-\rho_{F_0}(x)}<\epsilon$.

    Let $\omega_0\in F_0$ be such that $\rho_{F_0}(x)=\abs{\omega_0(x)}$.
    Let $V$ be the weak-$*$ open neighborhood of $\omega_0$ defined by 
    \[
    V=\{\omega\in (E^*)_1\colon\abs{\omega(x)-\omega_0(x)}<\epsilon\}.
    \]
    Then for all $\omega\in V$ we have $\abs{\omega(x)}>\abs{\omega_0(x)}-\epsilon$, so by definition for all $F\in \mathcal I_V$ we have $\rho_F(x)>\rho_{F_0}(x)-\epsilon$. 
    
    Towards the other inequality, consider the following larger weak-$*$ open neighborhood of $\omega_0$: 
    \begin{align*}
    W&=\{\omega\in (E^*)_1\colon \exists \omega'\in F_0, \abs{\omega'(x)-\omega(x)}<\epsilon\}\\&=\bigcup_{\omega'\in F_0}\{\omega\in (E^*)_1\colon \abs{\omega'(x)-\omega(x)}<\epsilon\}.
    \end{align*}
    Then $F_0\subseteq W$, and it is straightforward to check that for all $F\in \mathcal C_W$, we have
    $\rho_F(x)<\rho_{F_0}(x)+\epsilon$.
    
    We conclude that for all $F$ belonging to the open neighborhood $\mathcal U=\mathcal I_V\cap \mathcal C_W$ of $F_0$, we have 
    $\abs{\rho_F(x)-\rho_{F_0}(x)}<\epsilon$ as wanted.
\end{proof}
\begin{rmk}
The proof actually shows that $F\mapsto \rho_F(x)$ is lower semi-continuous if we endow $\mathcal F(X)$ with the lower topology, and upper semi-continuous if we endow $\mathcal F(X)$ with the upper Vietoris topology.
\end{rmk}

We will now use compactness so as to upgrade the previous lemma on a smaller set of closed subsets, namely convex \emph{balanced} closed subsets. Recall that a subset $F$ of $E^*$ is \textbf{balanced} if for every $\lambda\in(\C)_1$ and every $x\in F$, we have $\lambda x\in F$.
Denote by $\convexbalanced((E^*)_1)$ the space of balanced closed convex
subsets of $(E^*)_1$. We have the following well-known proposition.

\begin{prop}\label{prop: convexbalanced is compact}
    The space $\convexbalanced((E^*)_1)$ is a compact Hausdorff subspace of $\mathcal F((E^*)_1)$ for the Vietoris topology.
\end{prop}
\begin{proof}
    By Theorem \ref{thm: Vietoris is compact Hausdorff} the space $\mathcal F((E^*)_1)$ is compact and Hausdorff, so it suffices for us to show that $\convexbalanced((E^*)_1)$ is closed in $\mathcal F((E^*)_1)$.

    To this end, we first show that closed balanced subsets of $(E^*)_1$ form a closed subspace of $\mathcal F((E^*)_1)$. 
    
    By Lemma \ref{lem: belonging is closed}, the set of all $F\in\mathcal F((E^*)_1)$ such that $0\in F$ is closed.
    Now, for for every non-zero $\lambda\in(\C)_1$ consider the homeomorphism $m_\lambda:(E^*)_1\to \abs{\lambda}(E^*)_1$ defined by $m_\lambda(\omega)=\lambda\omega$. 
    It follows that $M_\lambda: \mathcal F((E^*)_1)\to\mathcal F(\abs{\lambda}(E^*)_1)\subseteq\mathcal F((E^*)_1)$ given by $F\mapsto \lambda F$ is a $\vietoristopo$ to $\vietoristopo$ homeomorphism, which yields that the subspace consisting of all 
    $F\in\mathcal F((E^*)_1)$ such that $M_\lambda(F)\subseteq F$ is closed by Lemma
    \ref{lem: closed inclusion vietoris}.
    The intersection over $\lambda\in(\C)_1$ of these closed subspaces with the closed subspace of closed subsets containing $0$ is the subspace of all balanced closed subsets, which is thus closed as wanted.
    
    We next show that convexity also defines a closed subspace, for which we
    need to use the lower topology.
    Namely, for every $t\in[0,1]$, consider the continuous map $f^t:(E^*)_1\times (E^*)_1\to (E^*)_1$ which takes $(\omega_1,\omega_2)$ to $t\omega_1+(1-t)\omega_2$. By Lemma \ref{lem: induced continuity for lower}, 
    it induces a continuous map 
    \[
    f^t_*:\mathcal F\left((E^*)_1\times(E^*)_1\right)\to \mathcal F(E^*)_1)  
    \]
    for the lower topologies. In particular the restriction of $f^t_*$ to 
    $$\mathcal F((E^*)_1)\times \mathcal F((E^*)_1)\subseteq \mathcal F((E^*)_1\times(E^*)_1)$$ is $\lowertopo\times\lowertopo$ to $\lowertopo$ continuous. Note that for all $A,B\in\mathcal F((E^*)_1)$, we have 
    $f^t_*(A,B)=tA+(1-t)B$.
    
    It follows that the map $$\Phi_t:F\mapsto (tF+(1-t)F, F)=(f^t_*(F,F),F)$$is $\vietoristopo$ to
    $\lowertopo\times\vietoristopo$ continuous. Since $tF+(1-t)F\subseteq F$ is equivalent to $F\in\Phi_t\inv(\{(F,G)\colon F\subseteq G\})$, and since 
    $\{(F,G)\colon F\subseteq G\}$ is $\lowertopo\times\vietoristopo$-closed by Lemma \ref{lem: closed inclusion vietoris},
    we conclude that the set of all $F\in \mathcal F((E^*)_1)$ such that $tF+(1-t)F\subseteq F$ is $\vietoristopo$-closed. Taking the intersection
    over all $t\in[0,1]$ of the spaces of such $F$'s, we arrive at the desired conclusion that the space of convex closed subsets of $(E^*)_1$ is closed 
    in $\mathcal F((E^*)_1)$. 
    
    Since the space $\convexbalanced((E^*)_1)$ is the intersection of the closed 
    space of balanced subsets with the closed space of convex subsets, it is also  closed as wanted.
\end{proof}

The following proposition is key to understanding the Vietoris topology on $\convexbalanced(E^*)$.

\begin{prop}\label{prop: preparation duality}
    Let $(E,\norm\cdot)$ be a normed vector space, then the map
    \[
    \begin{array}{rcc}
         \rho: \convexbalanced((E^*)_1)&\to&\R^E \\
         F&\mapsto& \rho_F
    \end{array}
    \]
    where $\rho_F$ is given by Equation \eqref{eq: definition of rhoF},
    is a homeomorphism onto its image.
\end{prop}
\begin{proof}
    By Lemma \ref{lem: pointwise continuity}, the map $\rho$ is continuous. Now, since $\convexbalanced((E^*)_1)$ is compact and $\R^E$ is Hausdorff, it suffices for us to show $\rho$ is injective. 
    This is an immediate consequence of the fact that 
    for all $F\in\convexbalanced((E^*)_1)$, we have that 
    \[
    F=\{\omega\in (E^*)_1\colon \forall x\in E,\, \abs{\omega(x)}\leq \rho_F(x)\}.
    \]
    To see why this holds, note that the inclusion $F\subseteq\{\omega\in (E^*)_1\colon \forall x\in (E)_1, \,\abs{\omega(x)}\leq \rho_F(x)\}$ follows from the definition of $\rho_F$. 
    Conversely, suppose $\omega_0\notin F$, then the Hahn-Banach theorem \cite[Thm.~3.7]{rudinFunctionalAnalysis2007} and duality \cite[Thm.~3.10]{rudinFunctionalAnalysis2007} grant us $x\in E$
    such that $\omega_0(x)>1$ but $\abs{\omega(x)}\leq 1$ for all $\omega\in F$.
    It follows that $\rho_F(x)\leq 1$, but then $\omega_0(x)>1\geq\rho_F(x)$ so 
    $\omega_0\notin\{\omega\in (E^*)_1\colon \forall x\in E, \abs{\omega(x)}\leq \rho_F(x)\}$, as wanted. 
\end{proof}

\subsection{Vietoris topology on grassmanian of dual spaces}

Given a normed vector space $(E,\norm\cdot)$, let us denote by $\mathcal G_{w*}(E^*)$
the set of all weak-$*$ closed subspaces of $E^*$, which we call the \textbf{grassmanian} of $E^*$.
Note that any subspace of a normed space is completely determined by its unit ball which is convex and balanced.
Thus, we may view $\mathcal G_{w*}(E^*)$ as a subspace of the compact Hausdorff space 
$\convexbalanced((E^*)_1)$ and we endow it with the induced topology, 
as Maréchal does at the beginning of her paper \cite{marechalTopologieStructureBorelienne1973}.

However, contrary to what 
she states there, $\mathcal G_{w*}(E^*)$ is \textit{not} compact in general, or equivalently it can \textit{fail} to be closed in $\convexbalanced((E^*)_1)$, even in the case when $E^*$ is a von Neumann algebra.
Here is a counterexample.

\begin{countex}\label{countex: not compact}
Consider the von Neumann algebra $\ell^\infty(\N)=(\ell^1(\N))^*$, which is generated by $(\delta_n)_{n\geq 1}$ where $\delta_n(m)=0$ if $n\neq m$ and $\delta_n(m)=1$ if $n=m$. 
For every $n\geq 1$, let $V_n=\C(\frac 12\delta_0+\delta_n)$, whose unit ball is 
$$B_n=(\C)_1\cdot (\frac 1{2} \delta_0+\delta_n).$$
Observe that the map $\xi\in(\ell^\infty(\N))_1\mapsto (\C)_1\xi$ is continuous if we endow $(\ell^\infty(\N))_1$ with the weak-$*$ topology and $\mathcal F((\ell^\infty(\N))_1)$ with the associated Vietoris topology. 
Since $\delta_i\to 0$ for the weak-$*$ topology, we have 
$$\lim_{n\to+\infty}B_n= (\C)_1\cdot \frac 1{2}\delta_0,$$
which is not the unit ball of a weak-$*$ closed vector subspace of $\ell^\infty(\N)$.
\end{countex}


We will see later on that $\mathcal G_{w*}(E^*)$ is however Polish when $E$ is separable (see Corollary \ref{cor: weakstar grassmanian is Polish}), 
which is sufficient to fix Maréchal's proof.


Proposition \ref{prop: preparation duality} allows us to put the equivalence between (ii) and (iii) in \cite[Thm.~2.8]{haagerupEffrosMarechalTopologySpace1998} in the proper context, where for $\omega\in E^*$ and $x\in E$ we let $\la \omega,x \ra=\omega(x)$.

\begin{prop}\label{PropConvergenceSubspace}
    Let $(E,\norm\cdot)$ be a normed vector space. Let $(V_i)_{i\in I}$ be a net consisting of elements of $\mathcal G_{w*}(E^*)$, let $V\in \mathcal G_{w*}(E^*)$. Then the following are equivalent:  
    \begin{enumerate}[(i)]
        \item $(V_i)_1\to(V)_1$ in the Vietoris topology;
        \item for all $x\in E$, we have 
        \[
        \norm{\la\cdot , x\ra}_{\restriction V_i}\to \norm{\la\cdot,x\ra}_{\restriction V}.
        \]
    \end{enumerate}
\end{prop}
\begin{proof}
    Recall that in Proposition \ref{prop: preparation duality} we showed that the map $\rho_\cdot:F\in\mathcal \convexbalanced((E^*)_1)\to \R^E$ defined by 
    $\rho_F(x)=\sup_{\omega\in F} \abs{\omega(x)}$ is a homeomorphism onto its image. When $F=(V)_1$ for a weak-$*$ closed subspace $V$, we have
    \[
    \rho_{(V)_1}(x)=\sup_{\omega\in (V)_1}\abs{\omega(x)}=
    \sup_{\omega\in (V)_1}\abs{\la\omega,x\ra}=\norm{\la\cdot,x\ra}_{\restriction V},
    \]
    which finishes the proof by Proposition \ref{prop: preparation duality}.
\end{proof}

\subsection{Wijsman topology on grassmanian of Banach spaces}

Given a normed vector space $(E,\norm\cdot)$, denote by $\mathcal G_{\norm\cdot}(E)$ its space of norm-closed subspaces,
which is naturally a subset of $\mathcal F^*(E)$. 
Denoting by $d_{\norm\cdot}$ the natural metric associated to the norm (given by $d_{\norm\cdot}(u,v)=\norm{u-v}$), we endow $\mathcal G_{\norm\cdot}(E)$
 with the topology induced by the Wijsman topology associated to the metric $d_{\norm\cdot}$ on $\mathcal F^*(E)$, as defined in Section \ref{sec: wijsman}.

\begin{prop}\label{prop: subspace is Wijsman closed}
    Let $(E,\norm\cdot)$ be a normed vector space. Then the set  $\mathcal G_{\norm\cdot}(E)$ of all its closed vector subspaces is closed in $\mathcal F^*(E)$ for the Wijsman topology associated to the metric $d_{\norm\cdot}$.
\end{prop}
\begin{proof}
    We use a similar approach as for the proof of Proposition \ref{prop: convexbalanced is compact}: we show that the three defining properties of linear subspaces (viz.\ containing zero, stability under multiplication by a non-zero scalar, stability under addition) define closed subsets of $\mathcal F^*(E)$ in the Wijsman topology.
    
    First, a closed subset $F\in\mathcal F^*(E)$ contains zero if and only if 
    $d(0,F)=0$, so by continuity of $F\mapsto d(0,F)$ the space all $F\in\mathcal F^*(E)$ containing zero is closed. 
    
    Next, for every non-zero scalar $\lambda$, the multiplication map $m_\lambda: x\in E\mapsto \lambda E$ scales the metric $d_{\norm\cdot}$ by a factor $\abs\lambda$, so it induces a homeomorphism 
    \[
    \begin{array}{rcc}
         m_{\lambda*}: \mathcal F^*(E)&\to&\mathcal F^*(E)  \\
          F&\mapsto& \lambda F
    \end{array}
    \] 
    for the Wijsman topology. 
    It then follows from Lemma \ref{lem: containement is closed for wijsman} that the space of all $F\in\mathcal F^*(E)$ which are stable under multiplication by a non-zero scalar is closed for the Wijsman topology.
    
    Finally, since addition is $\norm\cdot$-continuous, it induces a continuous map $\mathcal F^*(E\times E)\to\mathcal F^*(E)$ for the lower topology by Lemma \ref{lem: induced continuity for lower}, in particular the map 
    \[
    \begin{array}{rcc}
        \mathcal F^*(E)\times\mathcal F^*(E)&\to&\mathcal F^*(E)\\
         (A,B)&\mapsto&\overline{A+B} 
    \end{array}
    \]
    is $\lowertopo\times\lowertopo$ to $\lowertopo$ continuous.
    Using Lemma \ref{lem: containement is closed for wijsman} one last time, we conclude by continuity that the space of all $F\in\mathcal F^*(E)$ such that $F+F\subseteq F$ is closed for the Wijsman topology, which finishes the proof.
\end{proof}

\begin{rmk}
The addition map $(A,B)\mapsto \overline{A+B}$ is not continuous for the Wijsman topology, 
even in dimension two (take a sequence $(A_n)$ of lines converging non-trivially to another line $B$). 
So in the above proof the lower topology is seemingly unavoidable.
\end{rmk}

\begin{rmk}
Similar arguments show for instance that the space 
of closed $C^*$-subalgebras of a given $C^*$-algebra $A$ is closed
for the Wijsman topology. 
So when $A$ is separable, the Wijsman topology induces a Polish topology on the space of closed subalgebras of $A$ by Theorem \ref{thm: Wijsman Polish}. 
It would be interesting to explore this topology further.
In particular, one could try to compute the exact Borel complexity of
various natural sets of subalgebras in this context, or try to understand the existence of dense orbits for the natural action of the automorphism group/unitary group of $A$ thereon.
\end{rmk}

\subsection{The duality theorem}

It is a well-known consequence of the Hahn-Banach theorem that in any normed space $(E,\norm\cdot)$ the polar map induces a bijection $\mathcal G_{\norm\cdot}(E)\to \mathcal G_{w*}(E^*)$, and here we observe that this map is a homeomorphism. 
Although this result is very natural, we have not been able to find it in the literature.

\begin{thm}\label{thm: duality}
    Let $(E,\norm\cdot)$ be a normed vector space. The natural map
\[
\begin{array}{rcl}
\Phi: \mathcal G_{\norm\cdot}(E)&\to&\mathcal G_{w*}(E^*)\\
 V&\mapsto&V^\perp\coloneqq\{\omega\in E^*\colon\forall x\in V,\, \omega(x)=0\}
\end{array}
\]
is a homeomorphism if we endow $\mathcal G_{\norm\cdot}(E)$ with the Wijsman topology and $\mathcal G_{w*}(E^*)$ with the Vietoris topology.
\end{thm}
\begin{proof}
The Hahn-Banach theorem ensures that $\Phi$ is a bijection whose inverse is the map $W\mapsto {}^\perp W\coloneqq\{x\in E\colon \forall \omega\in W,\, \omega(x)=0\}$.

Recall that by Proposition \ref{prop: preparation duality}, 
the map 
$W\in\mathcal G_{w*}(E^*)\mapsto \rho_W\in\R^E$ is a homeomorphism onto its image, where 
$\rho_W(x)=\sup_{\omega\in(W)_1}\abs{\omega(x)}$.
Now observe that for any $x\in E$ and $V\in \mathcal G_{\norm\cdot}(E)$, we have
\begin{align*}
d_{\norm\cdot}(x,V)&=\norm{ x+V}_{E/V}\\
&=\sup_{\omega\in (E/V)^*_1}\abs{\omega(x+V)}\\
&=\sup_{\omega\in (V^\perp)_1}\abs{\omega(x)}.\\
&=\rho_{V^\perp}(x).
\end{align*}
By the definition of the Wijsman topology and Proposition \ref{prop: preparation duality}, this finishes the proof.
\end{proof}

We can now fix the gap in Maréchal's proof, replacing
her erroneous statement that the grassmanian is compact for the Vietoris topology (see Counterexample \ref{countex: not compact}) by the fact that it is Polish, which
suffices to have the rest of her proof go through.

\begin{cor}\label{cor: weakstar grassmanian is Polish}
    Let $(E,\norm\cdot)$ be a separable Banach space.
    Then the grassmanian $\mathcal G_{w*}(E^*)$ of weak-$*$ closed subspaces of $E^*$ is a Polish space for the
    Vietoris topology. 
\end{cor}
\begin{proof}
    By the previous result, $\mathcal G_{w*}(E^*)$ is homeomorphic to $\mathcal G_{\norm\cdot}(E)$ endowed
    with the Wijsman topology.
    But the latter is a Polish space since it is a closed subspace of $\mathcal F^*(E)$ by Proposition \ref{prop: subspace is Wijsman closed} and $\mathcal F^*(E)$ is Polish for the
    Wijsman topology by Theorem~\ref{thm: Wijsman Polish}. 
    Being homeomorphic to a Polish space, $\mathcal G_{w*}(E^*)$ is Polish as wanted.
\end{proof}

\begin{rmk}
This corollary can also be proved directly, by showing that $\mathcal G_{w*}(E^*)$ is $G_\delta$ in the compact Polish space 
$\convexbalanced((E^*)_1)$. To be more precise, one shows that $B\in\convexbalanced((E^*)_1)$ is the unit ball of a weak-$*$ closed subspace of $E^*$ iff for all $s\in(0,1)\cap \Q$ we have $\displaystyle\frac 1s\left(  (E^*)_s\cap B\right)\subseteq B$. It can be shown that the intersection map $(F,G)\mapsto (F\cap G)$ is Baire class $1$, so for a fixed $s\in(0,1)\cap \Q$,  the condition $\displaystyle\frac 1s\left(  (E^*)_s\cap B\right)\subseteq B$ defines a $G_\delta$ set, which is not closed  in general as can be seen through a slight modification of Counterexample \ref{countex: not compact}.
\end{rmk}

\section{Michael's selection theorem}\label{sec: michael}

\subsection{Partitions of unity}\label{sec: Polish is paracompact}

We require a  definition: given an open cover $\mathcal U$ of a topological space $X$,
a \textbf{partition of unity subordinate to $\mathcal U$} is a family of 
functions $(\rho_i)_{i\in I}$ where 
\begin{enumerate}
	\item \label{cond:partunit is continuous}$\rho_i:X\to[0,1]$ is continuous; 
	\item \label{cond:partunit is loc finite}for every $x_0\in X$, 
	there is a neighborhood $V$ of $x_0$ such that there are only finitely many $\rho_i$'s whose 
	support intersect $V$;
	\item \label{cond:partunit sums to one}for all $x\in X$, $\sum_{i\in I}\rho_i(x)=1$;
	\item \label{cond:partunit refines cover}for every $i\in I$, the set $\supp \rho_i$ is contained in some $U\in \mathcal U$. 
\end{enumerate}
The direct proof of the following lemma is inspired by the paper \cite{dydakPartitionsUnity2003}.
\begin{lem}\label{lem:partition unity}
	Let $X$ be a Polish space. For every open cover $\mathcal U$ of $X$, there is a 
	partition of unity subordinated to $\mathcal U$. 
\end{lem}
\begin{proof}
	Since $X$ is Polish, by Lindelöf's lemma we may as well assume $\mathcal U$ is countable,
	and we thus write it as $\mathcal U=(U_n)_{n\in\N}$. Fix a compatible metric $d$ on $X$.
	As a first step, consider the sequence of functions $g_n$ given by
	$g_n(x)=\frac 1{2^n}\min(1,d(x, X\setminus U_n))$. Each $g_n$ is supported on $U_n$. 
	Moreover, each $g_n$ is $1$-Lipschitz, so if we define $g(x)=\sup_n g_n(x)$, the
	function $g$ is $1$-Lipschitz as well, in particular it is continuous. 
	Since $(U_n)$ is a cover of $X$, we have $g(x)>0$ for all $x\in X$. 
	
	For each $n\in \N$, consider the continuous function $f_n$ defined by 
	$$f_n(x)=\max\left(0, g_n(x)-\frac{g(x)}2\right)$$
	and observe that for all $x\in X$ there is $n\in\N$ such that $f_n(x)>0$.
	Consider $f(x)=\sum_{n\in\N} f_n(x)$, then $f$ takes only non-zero positive values and $f$ is continuous
	by absolute convergence (indeed $\norm{f_n}_{\infty}\leq 2^{-n}$ by construction).
	
	We finally let $\rho_n(x)=\frac{f_n(x)}{f(x)}$ and claim that $(\rho_n)$ is the desired partition 
	of unity subordinate to $\mathcal U$. Conditions \eqref{cond:partunit is continuous}, 
	\eqref{cond:partunit sums to one} and \eqref{cond:partunit refines cover} easily follow from the
	construction. Let us check that condition \eqref{cond:partunit is loc finite} also holds: pick $x_0\in X$, 
	since $g$ is continuous we find a neighborhood $V$ of $x_0$ such that $g(x)>\frac{g(x_0)}2$
	for all $x\in V$. Fix $N$ large enough so that $2^{-N}< \frac{g(x_0)}4$. Then for 
	all $x\in V$ and all $n\geq N$ we have $g_n(x)\leq 2^{-N}\leq \frac{g(x)}2$ so $f_n(x)=0$.
	So for all $n\geq N$ the support of $\rho_n$ is disjoint from $V$, and we conclude that
	condition \eqref{cond:partunit is loc finite} holds as well, thus finishing the proof.	
\end{proof}

\subsection{Statement and proof of Michael's selection theorem}

We will now state and prove Michael's selection theorem \cite{michaelContinuousSelections1956} 
in a version which is inspired by 
\cite[Sec.~3.8]{sakaiGeometricAspectsGeneral2013}.
Given a normed vector space $(E,\norm{\cdot})$, recall that we denote by 
$d_{\norm\cdot}$ the associated metric, defined by $d_{\norm\cdot}(v,w)=\norm{v-w}$.

Given $C\subseteq E$ closed, we denote by $\closedconvexstar(C)$ the space of non-empty convex closed subsets of $C$, and it is endowed with the lower topology. Given a topological space $X$, let us say that a continuous map $X\to \closedconvexstar(C)$ is \textit{lower continuous} if it is continuous for the lower topology. 

\begin{thm}[Michael]\label{thm: michael selection}
	Let $(E,\norm\cdot)$ be a normed vector space and let $X$ be a Polish space.
	Let $C$ be a convex subset of $E$ which is $d_{\norm\cdot}$-complete.
	Suppose $F: X\to \closedconvexstar(C)$ is a lower continuous map. 
	Then there is a continuous selection for $F$, namely a continuous map $f: X\to C$ such that for all $x\in X$, $f(x)\in F(x)$. 
\end{thm}

\begin{rmk}
Note that in our statement of Michael's selection theorem we \emph{do not} require the metric $d_{\norm\cdot}$ to be complete on $E$, but only when restricted to $C$. 
This is important since
in both our applications, the norm will actually not be complete on $E$.
\end{rmk}

The proof will use an approximate version of the result which we state and prove first.
\begin{lem}
	Let $(E,\norm\cdot)$ be a normed vector space, let  $C$ be a convex subset of $E$ which is $d_{\norm{\cdot}}$-complete.
	Suppose $F: X\to \closedconvexstar(C)$ is a lower continuous map, and let $\epsilon>0$.
	Then there is a continuous map $f: X\to C$ such that for all $x\in X$, $d_{\norm\cdot}(f(x),F(x))<\epsilon$. 
\end{lem}
\begin{proof}
    Given $v\in E$ and $R>0$, let $B(v,R)\coloneqq\{w\in E: \norm{v-w}<R\}$ be the open ball of radius $R$ centered at $v$.
	For each $v\in C$, define 
	\[
	U_v\coloneqq\{ x\in X: F(x)\cap B(v,\epsilon)\neq \emptyset\},
	\]
	then by lower continuity and the fact that $\closedconvexstar(C)$ consists of nonempty sets, $(U_v)_{v\in C}$ is an open cover of $X$.	
	Using Lemma \ref{lem:partition unity}, we find a partition
	of unity $(\rho_i)_{i\in I}$ subordinate to $\mathcal U$.
	We then have for all $i\in I$, some $v_i\in C$ such that for all $x\in\supp \rho_i$, one has $F(x)\cap B(v_i,\epsilon)\neq\emptyset$, which is equivalent
	to $d_{\norm\cdot}(v_i,F(x))<\epsilon$.
		
	We claim that the map 
	$f: x\mapsto \sum_{i\in I} \rho_i(x)v_i$ is as desired. Indeed, note that by the partition of unity conditions
	\eqref{cond:partunit is continuous} and \eqref{cond:partunit is loc finite}, the map $f$ is continuous, and that by condition \eqref{cond:partunit sums to one}
	$f(x)$ is a convex combination of finitely many elements of the convex set $(F(x)+B(0,\epsilon))\cap C$, hence it belongs to $C$ and satisfies $d(f(x),F(x))<\epsilon$.
\end{proof}

\begin{proof}[Proof of Michael's selection theorem] 
	We will use the previous lemma and Lemma \ref{lem: continuity of intersection with fattenings} in order 
	to inductively build a sequence of approximate selection which will be Cauchy for the uniform metric
	on the space of continuous maps from $X$ to $C$.
	
	Using the previous lemma, we start with a map $f_1: X\to C$ such that for all $x\in X$,
	we have $d(f_1(x),F(x))<1$. Let us show that we can then inductively build continuous maps $f_k: X\to C$
	such that for all $x\in X$, we have
	$d(f_{k}(x),F(x))<2^{-k-1}$ and $d(f_{k+1}(x),f_k(x))\leq{2^{-k}}$.

	Indeed, assuming $f_k$ has already been built, we have $d(f_k(x),F(x))<2^{-k-1}$ for all $x\in X$. 
	Towards applying Lemma \ref{lem: continuity of intersection with fattenings}, 
	let 
	\[
	W=\{(v_1,v_2)\in C\times C\colon \norm{v_1-v_2}<2^{-k-1}\},
	\]
	observe that $W_v$ is the open ball $B(v,2^{-k-1})$.
	Since $W$ is open, by Lemma \ref{lem: continuity of intersection with fattenings} the
	map $x\mapsto \overline{F(x) \cap W_{f_k(x)}}=\overline{F(x)\cap B(f_k(x),2^{-k-1})}$ is continuous. 
	By the previous lemma, we find a continuous map $f_{k+1}: X\to C$ such that for all $x\in X$, one has
	$d\left(f_{k+1}(x),\overline{F(x)\cap B(f_k(x),2^{-k-1})}\right)<2^{-k-2}$. Then $f_{k+1}(x)$ is as wanted: 
	its distance to $F(x)$ is less than $2^{-k-2}$ and it is $2^{-k-1}+2^{-k-2}<2^{-k}$-close to $f_k(x)$.
	
	Now by construction the sequence $(f_k)$ is uniformly Cauchy and consists of continuous functions
	taking values in the complete metric space $(C,d_{\norm\cdot})$, 
	so it has a continuous limit which
	is the desired map $f$.
\end{proof}

The following upgraded version of the selection theorem will be our main tool for applications. 

\begin{cor}\label{cor: michael dense selection}
	Let $(E,\norm\cdot)$ be a normed vector space and let $X$ be a Polish space. Let $C$ be a convex subspace of $E$ which is $d_{\norm\cdot}$-complete. Suppose additionally that $C$ is separable.
	
	Let $F: X\to \closedconvexstar(C)$ be a lower continuous map. Then there is a sequence of continuous map $f_n: X\to C$ such that for all $x\in X$, 
	$$F(x)=\overline{\{f_n(x): n\in\N\}}.$$
\end{cor}
\begin{proof}
	Let $(v_n)$ be dense in $C$ and let $f$ be a continuous selection for $F$ as provided by the previous theorem. 
	For every $n,m\in\N$ consider the open subset  $U_{n,m}\coloneqq\{x\in X: F(x)\cap B(v_n,\frac1{m})\neq \emptyset\}$. Since $X$ is Polish, all its open subsets are $F_\sigma$ and so we may write each $U_{n,m}$ as
	\[
	U_{n,m}=\bigcup_{p\in\N} C_{n,m,p},
	\]
	where each $C_{n,m,p}\subseteq U_{n,m}$ is closed. 
	Now for each $n,m,p$ define $F_{n,m,p}: X \to \closedconvexstar(C)$ by
	\[
	F_{n,m,p}(x)=
	\left\{\begin{array}{cl}F(x) & 
	\text{ if } x\not\in C_{n,m,p}; \\
	\,\,\\
	\overline{F(x)\cap B(v_n,1/m)} & 
	\text{ if }x\in C_{n,m,p}. \end{array}\right.
	\]
	We claim that each $F_{n,m,p}$ is still lower continuous. Indeed, consider a subbasic open set of the form $\mathcal I_O=\{F\in\closedconvexstar(C)\colon F\cap O\neq \emptyset\}$ where $O\subseteq C$ is open. We will show that $F_{n,m,p}\inv(\mathcal I_O)$ is open. 
	
	Let $x_0\in F_{n,m,p}\inv(\mathcal I_O)$. We have two cases to consider.
	\begin{itemize}
		\item If $x_0\notin C_{n,m,p}$ then since $C_{n,m,p}$ is closed, the set $(X\setminus C_{n,m,p})\cap F\inv(\mathcal I_O)$ is an open neighborhood of $x_0$ contained in $F_{n,m,p}\inv(\mathcal I_O)$ since the maps $F$ and $F_{n,m,p}$ coincide outside of $C_{n,m,p}$ and $F$ is lower continuous.
		\item If $x_0\in C_{n,m,p}$ then by definition we have $\overline{F(x_0)\cap B(v_n,1/m)}\cap O\neq \emptyset$ so  $F(x_0)\cap B(v_n,1/m)\cap O\neq \emptyset$. 
		The set $F\inv(\mathcal I_{B(v_n,1/m)\cap O})$ is then a neighborhood of $x_0$ all whose elements $x$ satisfy  $F_{n,m,p}(x)\in\mathcal I_O$.
	\end{itemize}
	So each $F_{n,m,p}$ is indeed lower continuous.
	The conclusion now follows by applying the previous theorem to each map $F_{n,m,p}$.
\end{proof}
\begin{rmk}
Our hypothesis on $X$ in Theorem \ref{thm: michael selection} is not optimal, and could be replaced by asking that $X$ is paracompact. Indeed Section \ref{sec: Polish is paracompact} shows that every Polish space is paracompact, which is the only property that is used in the proof. 
Similarly, Corollary \ref{cor: michael dense selection} holds as long as $X$ is paracompact and perfectly normal, and we refer the interested reader to \cite{sakaiGeometricAspectsGeneral2013} for details on these topological properties.
\end{rmk}
\subsection{Two examples of norms}\label{sec: ex of norms}
We now build the two norms that will allow us to use Michael's selection 
theorem in the framework of von Neumannn subalgebras: 
the first is for the ultraweak topology and will be useful towards proving
Maréchal's theorem,
and the second is for the strong-$*$ topology and will allow us to easily 
deduce the strong-$*$ selection theorem of Haagerup and Winsløw.

\begin{lem}\label{lem: norm for weak}
Let $\mathcal H$ be a separable Hilbert space. There is a norm on $\bh$ whose restriction to $\bh_1$ induces a complete metric which is compatible with the ultraweak topology.
\end{lem}
\begin{proof}
    Recall that on $\BH_1$, the weak and the ultraweak topology coincide (see e.g.\ \cite[Prop.~4.6.14]{pedersenAnalysisNow1989}). Moreover $\BH_1$ is ultraweakly compact by the Banach-Alaoglu theorem, so it suffices to find a norm $\rho$ which induces the weak topology on $\BH_1$ (indeed the associated metric $d_{\rho}$ will automatically be complete on $\BH_1$ by compactness).

    Let us fix a sequence $(\xi_m)$  dense in the unit ball of $\mathcal H$. 
    Let $m\mapsto (k_m,l_m)$ be a surjection $\N\to\N\times\N$. Then consider the semi-norm
	$$\rho(x)=\sum_{m\in\N} \frac 1{2^m}\abs{\la x\xi_{k_m},\xi_{l_m}\ra}, \quad x\in \BH.$$
	Observe that if $(x_i)$ is a uniformly bounded net converging to zero weakly, then 
	$\rho(x_i)\to 0$, in particular $\rho$ is refined by the weak topology on the unit ball.
	
	Let us show that $\rho$ does induce the weak topology on the unit ball, which will imply in particular that it is a norm. Let $x\in\BH_1$ and assume that we have a sequence $(x_n)$ in the unit ball such that $\rho(x_n-x)\to 0$. Let $\xi,\eta\in \mathcal H$, without loss of generality be both non-zero and we thus consider the associated unit vectors $\xi'\coloneqq \frac 1{\norm\xi}\xi$ and $\eta'\coloneqq \frac 1{\norm\eta}\eta$. For any $k, l,n \in\N$ we have
    \[
    \la (x_n-x)\xi',\eta'\ra=\la (x_n-x)\xi_k,\xi_l\ra+\la (x_n-x)(\xi'-\xi_k),\xi_l\ra+ \la (x_n-x)\xi',\eta'-\xi_l\ra.
    \]
	Since $\rho(x_n-x)\to 0$ the first term tends to zero, while the two other terms can be made arbitrarily small by the Cauchy-Schwartz inequality, the density of $(\xi_k)$, and the fact that $\norm{x_n-x}\leq 2$. 
	So we do have $\la (x_n-x)\xi',\eta'\ra\to 0$, and
	multiplying by $\norm{\xi}\norm\eta$ we obtain $\la (x_n-x)\xi,\eta\ra\to 0$ as wanted. This finishes the proof of the lemma.
\end{proof}

\begin{lem}\label{lem: norm for strong star}
    Let $\mathcal H$ be a separable Hilbert space.
    There is a norm on $\bh$ whose restriction to $\bh_1$ induces a complete metric which is compatible with the strong-* topology.
\end{lem}
\begin{proof}
    Let us again fix a sequence $(\xi_m)$ dense in the unit ball of $\mathcal H$.
    Define the semi-norm
	$$\rho(x)=\sum_{m\in\N} \frac 1{2^m}\left(\norm{x\xi_{m}}+\norm{x^*\xi_{m}}\right),\quad x\in \BH.$$
	Again $\rho$ is clearly refined by the strong-* topology on $\BH_1$. Moreover if $(x_n)$ is a sequence in $(\BH)_1$ and if $\rho(x_n)\to 0$, then  for all $m$ we have both $\norm{x_n \xi_m}\to 0$ and $\norm{x_n^*\xi_m}\to 0$. We first have to show that $\rho$ induces the strong-* topology on $\BH_1$, which will in particular imply it is a norm.
	
    Let $x\in\BH_1$ and assume that we have a sequence $(x_n)$ in the unit ball such that $\rho(x_n-x)\to 0$.
	As before we take $\xi\in \mathcal H$ nonzero and let $\xi'=\frac 1{\norm\xi}\xi$. 
	By the triangle inequality we have for all $n,m\in\N$
	\begin{align*}
	    \norm{(x_n-x)\xi'}&\leq \norm{(x_n-x)\xi_m}+\norm{(x_n-x)(\xi'-\xi_m)}\\
	    &\leq \norm{(x_n-x)\xi_m}+2\norm{\xi'-\xi_m}.
	\end{align*}
	The second term can be made arbitrarily small by density, while when $m$ is fixed the first term tends to zero because $\rho(x_n-x)\to 0$; so we conclude that $\norm{(x_n-x)\xi'}\to 0$, and hence $\norm{(x_n-x)\xi}\to 0$ as wanted.
	Replacing $x_n$ by $x_n^*$ and $x$ by $x^*$ in the above argument, we obtain the other desired convergence $\norm{(x_n^*-x^*)\xi}\to 0$, which finishes the proof that $\rho$ induces the strong-$*$ topology on $\BH_1$.

	We finally show that the metric $d_\rho$ given by $\rho$ is complete on $\BH_1$.
	Let $(x_n)$ be a $d_\rho$-Cauchy sequence. Then for all $m\in\N$, the sequences $(x_n\xi_m)$ and $(x_n^*\xi_m)$ are Cauchy in $\h$ and hence they are norm convergent.
	
	By weak compactness of $\BH_1$ and weak continuity of the adjoint, there is a subsequence $(n_k)$ and $x\in\BH_1$ such that for all $\xi\in\H$,  $x_{n_k}\xi\rightarrow x\xi$ and $x_{n_k}^*\xi\rightarrow x^*\xi$ weakly in $\h$ as $k\rightarrow\infty$. 
	Therefore, the limit in norm of $(x_n\xi_m)$ and $(x_n^*\xi_m)$ respectively must agree with $x\xi_m$ and $x^*\xi_m$ for all $m\in\N$. Since $(x_n)$ is bounded a density argument shows that $x_n\rightarrow x$ and $x_n^*\rightarrow x^*$ strongly. 
\end{proof}
\section{Applications}\label{sec: applications}

\subsection{Proof of Maréchal's theorem}

We begin by reproving Maréchal's selection result for the weak (equivalently, ultraweak) topology.

\begin{prop}[{\cite[Prop.~1]{marechalTopologieStructureBorelienne1973}}]
\label{prop:weakly continous choice in the unit ball}
	Let $M$ be a  von Neumann algebra with separable predual. Let $\mathcal G_{w*}(M)$ denote its space of ultraweakly closed subspaces, endowed with the Vietoris topology, and endow the unit ball $(M)_1$ with the weak topology.
	There is a sequence of continuous maps $x_n:\mathcal G_{w*}(M)\to (M)_1$ such that 
    for all $V\in\mathcal G_{w*}(M)$, the sequence $(x_n(V))_{n\in\N}$ is dense in $(V)_1$ for the weak topology. 
\end{prop}
\begin{proof}
    Recall from Corollary \ref{cor: weakstar grassmanian is Polish} that $\mathcal G_{w*}(M)$ is a Polish space for the Vietoris topology, which refines the lower topology. Since Lemma \ref{lem: norm for weak} provides us a norm on $M$ which when restricted to $(M)_1$ yields a complete metric inducing the ultraweak topology, the conclusion follows directly from the dense version of Michael's selection theorem (Corollary \ref{cor: michael dense selection}).
\end{proof}

\begin{rmk}
Maréchal's result is actually more precise since she shows that the sequence of maps $(x_n)$ is \emph{equicontinuous}.
This allows her to completely recast continuity in terms of selections (see \cite[Prop.~3]{marechalTopologieStructureBorelienne1973}). 
It would be interesting to see if a version of Michael's selection theorem can encompass this statement, possibly by requiring the source space $X$ to be compact.
\end{rmk}

We can now finish the proof exactly as Maréchal did in \cite[Cor.~2]{marechalTopologieStructureBorelienne1973}.
Given a von Neumann algebra $M$, we denote by $\mathcal S(M)$ its space
of von Neumann subalgebras (with the same unit as $M$). 
Since such subalgebras are ultraweakly closed,
we have $\mathcal S(M)\subseteq \mathcal G_{w*}(M)$ and we thus endow it with the Vietoris topology.

\begin{thm}\label{thm:polish_vN}
	Let $M$ be a von Neumann algebra with separable predual. 
	The space $\mathcal S(M)$ of von Neumann subalgebras of $M$ is $G_\delta$ inside the Polish space $\mathcal G_{w*}(M)$. In particular, it is a Polish space. 
\end{thm}
\begin{proof}
	Note that an ultraweakly closed subspace $V$ of $M$ is a von Neumann algebra iff its closed unit ball $(V)_1$ satisfies the following three additional conditions: 
	\begin{enumerate}[(i)]
		\item \label{cond: adjoint closed}$(V)_1$ is closed under taking adjoints;
		\item \label{cond: cont id}$(V)_1$ contains $1_M$;
		\item \label{cond: N closed under multiplication}  $(V)_1$ is closed under multiplication.
	\end{enumerate}
	Condition \eqref{cond: adjoint closed} defines a closed subset of $\mathcal G_{w*}(M)$ since the adjoint map is a homeomorphism of $(M)_1$ for the weak topology. Condition \eqref{cond: cont id} defines a closed set by Lemma \ref{lem: belonging is closed}. Condition \eqref{cond: N closed under multiplication} is more subtle since the product map is not weakly continuous, even on bounded sets. We will actually show that condition \eqref{cond: N closed under multiplication} defines a $G_\delta$ set, which will conclude the proof since closed sets are $G_\delta$ and any countable (in particular, finite) intersection of $G_\delta$ subsets is itself $G_\delta$.

	Consider a sequence $(x_n)$ as provided by Proposition \ref{prop:weakly continous choice in the unit ball}: for every $n$ the map $x_n: \mathcal G_{w*}(M)\to (M)_1$ is continuous and for each $V\in\mathcal G_{w*}(M)$, the sequence $(x_n(V))_{n\in\N}$ is dense in $(V)_1$ for the weak topology. 
	Now consider the following set: 
	\begin{align*}
	    \mathcal P(M)&\coloneqq\{V\in \mathcal G_{w*}(M)\colon \forall n,m\in\N, x_n(V)x_m(V)\in V\}\\
	    &=\bigcap_{n,m\in\N}\{V\in \mathcal G_{w*}(M)\colon x_n(V)x_m(V)\in V\}.
	\end{align*}
	Because the multiplication map is Baire class 1 for the weak topology (cf. Proposition \ref{prop: mult is weakly baire class 1}) and the relation $\in$ defines a closed set (Lemma \ref{lem: belonging is closed}) and the maps $x_n$ are continuous, we have that $\mathcal P(M)$ is a countable intersection of $G_\delta$ subsets, hence it is $G_\delta$. Let us check that the elements of $\mathcal P(M)$ are exactly those which satisfy \eqref{cond: N closed under multiplication} so as to finish the proof. 
	
	Clearly any $V$ satisfying \eqref{cond: N closed under multiplication} will belong to $\mathcal P(M)$. 
	Conversely, if $V\in\mathcal P(M)$, for any $n\in\N$ the weak continuity of left multiplication and density of $(x_m(V))_{m}$ in the unit ball of $V$ yields that $x_n(V)\cdot(V)_1\subseteq (V)_1$. 
	Then the weak continuity of right multiplication (see Lemma \ref{lem: separate continuity mult weak}) and the density of $(x_n(V))_n$ in $(V)_1$ yield that for every $x\in (V)_1$, we have
	$(V)_1\cdot x\subseteq (V)_1$. We conclude that the unit ball of $V$ is closed under multiplication as wanted.
\end{proof}

\subsection{Proof of the Haagerup-Winsløw  selection theorem}

The following result is implicit in \cite{haagerupEffrosMarechalTopologySpace1998}, and is the key to 
our proof of their selection theorem.

\begin{prop} \label{prop: same lower topologies on marechal}
	Let $M$ be a von Neumann algebra.
	The lower topologies associated to the weak topology and to the strong-$*$ topology on the unit ball coincide on $\mathcal S(M)$. 
\end{prop}
\begin{proof}
	Of course the lower topology associated the weak topology is refined by the lower topology associated to the strong-* topology. 
	Conversely, we need to show that given a von Neumann subalgebra $N_0\subseteq M$, every lower subbasic strong-$*$ neighborhood of $N_0$ is a lower weak neighborhood of $N_0$. 
	
	To this end, fix a strong-$*$ open set $U$ intersecting $(N_0)_1$ non-trivially and consider the lower strong-$*$ subbasic neighborhood $\mathcal I_U$ of $N_0$. 
    Since we  will need to stay in the unit ball, the following version
    of the Kaplansky density theorem will be used.\\
    
    \begin{claim}
        There is a strong-$*$ to strong-$*$ continuous map $\phi:M\to(M)_1$ such that for all $x\in (M)_1$, $\phi(x)=x$.
    \end{claim}
    \begin{cproof}
        The map $\phi$ constructed in the proof of \cite[Lem.~2.2]{haagerupEffrosMarechalTopologySpace1998} works as required (to see that it is strong-$*$ to strong-$*$ continuous, one can for instance appeal to \cite[Lem.~44.2]{conwayCourseOperatorTheory2000}).
    \end{cproof}
	
	Now let $x_0\in U\cap(N_0)_1$, and fix a map $\phi$ as in the claim. Since $N_0$ is a von Neumann algebra, $x_0=\phi(x_0)$
	can be written as a linear combination of four unitaries 
	$$x_0=\lambda_1u_1+\lambda_2 u_2+\lambda_3u_3+\lambda_4u_4.$$
	By strong-$*$ continuity of addition and of $\phi$, we find for $i\in\{1,\dots,4\}$ 
	a strong-$*$ open set $U_i$ containing $u_i$ such that $\phi(\lambda_1U_1+\lambda_2U_2+\lambda_3U_3+\lambda_4U_4)\subseteq U\cap (M)_1$. 
	
	By Lemma \ref{lem: unitary have same weak and strong nbhd on B1} we find weakly open subsets $V_i$ of $(M)_1$ containing $u_i$ such that $V_i\subseteq U_i\cap(M)_1$. 
	Then $\bigcap_{i=1}^4 \mathcal I_{V_i}$ is a lower weak neighborhood of $N_0$
	contained in $\mathcal I_{U}$ as desired.
\end{proof}

We can finally state and prove the Haagerup-Winsløw  selection theorem, thus upgrading Proposition \ref{prop:weakly continous choice in the unit ball} in the context of $\mathcal S(M)$.

\begin{thm}[{\cite[Thm. 3.9]{haagerupEffrosMarechalTopologySpace1998}}]\label{thm:*strongly continous choice in the unit ball}
	Let $M$ be a von Neumann algebra with separable predual. 
	There is a sequence $(x_n)$ of strong-$*$ continuous maps $\mathcal S(M)\to (M)_1$ such that for all $N\in\mathcal S(M)$, the set $\{x_n(N)\colon n\in\N\}$ is dense in $(N)_1$ for the strong-$*$ topology. 
\end{thm}
\begin{proof}
    By Theorem \ref{thm:polish_vN} the space $\mathcal S(M)$ is Polish for the Vietoris topology associated to the weak topology on $(M)_1$, 
    which refines the lower topology associated to the weak topology. 
    Proposition \ref{prop: same lower topologies on marechal} implies that the latter coincides with the lower topology associated to the strong-$*$ topology on $(M)_1$. 
	By the dense version of Michael's selection theorem (Corollary \ref{cor: michael dense selection}), we now only need to show that we can endow $M$ with a norm  whose restriction to the unit ball of $M$ induces a complete metric compatible with the strong-$*$ topology, which is precisely the content of  Lemma \ref{lem: norm for weak}.\end{proof}

\subsection{The set of finite von Neumann algebras is \texorpdfstring{$\Pizerothree$}{Pi zero three}-complete}

As an addition to the numerous applications of the above theorem that one can find in \cite{haagerupEffrosMarechalTopologySpace1998,haagerupEffrosMarechalTopology2000}, we find an optimal refinement of a result of Nielsen showing that the set of finite von Neumann algebras is Borel \cite{nielsenBorelSetsNeumann1973}.

Recall that a \textbf{finite} von Neumann algebra $M$ with separable predual is characterised by the existence of a finite trace, namely a faithful normal state $\tau:M\to\C$ such that $\tau(ab)=\tau(ba)$ for all $a,b\in M$, on it. More generally, finite von Neumann algebras are defined by "finiteness" of the identity projection, which in terms of the comparison theory of projections, means that the identity operator is not (Murray-von Neumann) equivalent to any proper subprojection in $M$. We refer the reader to \cite[Sec.~6.3]{liIntroductionOperatorAlgebras1992} for more details on finite von Neumann algebras, including the proofs of equivalences between different characterisations of it.

The characterization that will prove useful to us relies on the existence of a compatible norm for the strong topology on the unit ball in the same spirit as those from Section
\ref{sec: ex of norms}.
We now fix one such norm, although some of the equivalent statements below would be more naturally stated in terms of uniform structures. Let $(\xi_n)$ be dense in $(\H)_1$, define for $x\in\BH$
$$\rho_s(x)=\sum_n \frac 1{2^n} \norm{x\xi_n}.$$
A straightforward simplification of the proof of Lemma \ref{lem: norm for strong star} yields that $\rho_s$ is a norm inducing a metric $d_s$ whose restriction to the unit ball is complete and induces the strong topology.

\begin{thm}\label{thm: chara finite}
Let $M$ be a von Neumann algebra with separable predual endowed with the aforesaid norm $\rho_s$ and let $d_s$ denote the associated metric. Then the following are equivalent:
\begin{enumerate}
    \item \label{item: M finite}$M$ is finite;
    \item \label{item: star is continuous on unit ball}$x\mapsto x^*$ is continuous for the strong topology on $(M)_1$;
    \item \label{item: star is continuous at zero}$x\mapsto x^*$ is continuous at $0$ for the strong topology on $(M)_1$;
    \item \label{item: star is unif continuous on unit ball}$x\mapsto x^*$ is uniformly continuous for $d_s$ on $(M)_1$.
\end{enumerate}
\end{thm}
\begin{proof}
The equivalence between \eqref{item: M finite} and \eqref{item: star is continuous on unit ball} follows from results of Sakai \cite{sakai1}, see also \cite[Thm.~6.3.12]{liIntroductionOperatorAlgebras1992}.
The fact that \eqref{item: star is continuous on unit ball}  implies  \eqref{item: star is continuous at zero} is clear.

Next, since $\rho_s$ is a norm and $x\mapsto x^*$ is antilinear, continuity at zero implies uniform continuity, so \eqref{item: star is continuous at zero} implies \eqref{item: star is unif continuous on unit ball}. 
The latter directly implies \eqref{item: star is continuous on unit ball}, which finishes the proof.
\end{proof}

\begin{thm}\label{thm: complexity finite}
    Let $\mathcal H$ be a separable Hilbert space.
    The set of finite von Neumann subalgebras in $\mathcal S(\BH)$ is $\Pizerothree$-complete.
\end{thm}
\begin{proof}
Denote by $\Subfin(\BH)$ the space of finite (unital) von Neumann subalgebras of $\BH$.
Fix a sequence of Maréchal to strong-$*$ continuous maps $x_n:\mathcal S(\BH)\to (\BH)_1$ as in Theorem \ref{thm:*strongly continous choice in the unit ball}, so that for every $M\in\mathcal S(\BH)$ $(x_n(M))$ is strong-$*$ dense in $(M)_1$.
Then by Theorem \ref{thm: chara finite}, $M\in\Subfin(\BH)$ iff for every $\epsilon>0$, 
there is $\delta>0$ such that 
\[
\text{for all }n,m\in\N\text{, if }d_s(x_n(M),x_m(M))<\delta\text{ then }d_s(x_n^*(M),x_m^*(M))\leq \epsilon.
\]
For every $n,m\in\N$ and $\delta,\epsilon>0$, let $\mathcal F_{n,m,\delta,\epsilon}$ denote the 
set of all $M\in\mathcal S(\BH)$ such that if $d_s(x_n(M),x_m(M))<\delta$, then $d_s(x_n^*(M),x_m^*(M))\leq \epsilon$. 
Equivalently, $\mathcal F_{n,m,\delta,\epsilon}$ is the set  of all $M\in\mathcal S(\BH)$ such that $d_s(x_n(M),x_m(M))\geq \delta$ or $d_s(x_n^*(M),x_m^*(M))\leq \epsilon$, so by continuity it is closed. So $\bigcap_{n,m} \mathcal F_{n,m,\delta,\epsilon}$ is closed as well.
By the above discussion
\[
\Subfin(\BH)=\bigcap_{\epsilon\in\Q^{>0}}\bigcup_{\delta\in\Q^{>0}}\bigcap_{n,m\in\N}\mathcal F_{n,m,\delta,\epsilon},
\]
so $\Subfin(\BH)$ is $\PPi^0_3$ as wanted.

To see that it is $\Pizerothree$-complete, let us recall
that the set $\Pfin(\N)^\N$ of sequences of finite subsets of $\N$ is a $\PPi^0_3$-complete subset of $(\{0,1\}^\N)^\N$ by Lemma \ref{lem: sequence of finite sets is Pi 0 3 complete}, so we will be done once we continuously reduce the latter set to $\Subfin(\BH)$.
To this end, we work on $\mathcal H=\ell^2(\N\times\N)$, which 
we freely identify with $\ell^2(\N)\otimes\ell^2(\N)$. Given a sequence $S=(S_n)_n\in(\{0,1\}^\N)^\N$ of subsets of $\N$, we map it to
$$f(S)\coloneqq\left(\bigoplus_n \mathcal B(\ell^2(\{n\}\times S_n))\right)\oplus\C(1-\pi_S)\subseteq \mathcal B(\mathcal H),$$
where $\pi_S\coloneqq\sum_ne_{nn}\ot p_{S_n}$, and  $e_{kl}=\la e_l,\cdot\ra e_k$ are the matrix units associated to the canonical orthonormal basis $(e_n)_{n\in\N}$ of $\ell^2(\N)$ and, for a subset $A\subseteq\N$, $p_{A}\coloneqq\sum_{k\in A}e_{kk}\in\mathcal{B}(\ell^2(\N))$ is the orthogonal projection onto $\ell^2(A)$. 
Clearly, $f((S_n)_n)$ is finite (as a von Neumann algebra) iff for all $n$ the subset $S_n$ is finite, thus it suffices to show that $f$ is continuous. 

By Proposition \ref{PropConvergenceSubspace}, the continuity of $f$ will be established once we show that for each linear functional $\omega\in\BH_*$, we have that the map $S\mapsto \norm{\omega_{\restriction f(S)}}$ is continuous.
Let us thus consider the set 
$$
\mathcal{Z}\coloneqq\{\omega\in\mathcal{B}(\mathcal H)_*\,:\
S\mapsto \Vert\omega_{\restriction f(S)}\Vert\text{ is continuous}\}.
$$ 
Since for any linear subspace $M$ of $\BH$, the map $\omega\mapsto \omega_{\restriction M}$ is $1$-Lipschitz, the set $\mathcal{Z}$ is norm closed, and we thus only have to show that $\mathcal Z$ is dense.

For any vectors $\xi,\eta\in\mathcal H$ and $x\in\BH$, let $\omega_{\xi,\eta}(x)=\la \xi,x \eta\ra$.
The set of linear functionals of the form $$\omega=\sum_{i,j=1}^n\omega_{e_i\ot\xi_i,e_j\ot\eta_j},$$ for $\xi_i,\eta_j\in \ell^2(\N)$, is norm dense in $\mathcal{B}(\mathcal H)_*$ (see e.g.~\cite[Thm.~1.1.9 and Prop.~1.1.7]{liIntroductionOperatorAlgebras1992}), and we will show that it is contained in $\mathcal Z$, thus finishing the proof. 

Fix a linear functional $\omega=\sum_{i,j=1}^n\omega_{e_i\ot\xi_i,e_j\ot\eta_j}$ as above, and let us compute $\Vert\omega_{\restriction f(S)}\Vert$ where $S=(S_n)_n\in(\{0,1\}^\N)^\N$. Note that the unit ball of $f(S)$ is exactly the set of  $x\in\mathcal{B}(\mathcal H)$ of the form 
$$
x=\sum_{n\in\N}e_{nn}\ot \iota_{S_n}x_np_{S_n}+\lambda(1-\pi_S),
$$
where $x_n\in\mathcal{B}(\ell^2(S_n))$ and $\lambda\in\C$ are such that $\Vert x_n\Vert\leq 1$ and $\vert\lambda\vert\leq 1$, and $\iota_{S_n}$ denotes the natural inclusion of $\ell^2(S_n)$ in $\ell^2(\N)$. For such $x$ one has 
$$
\omega(x)=\sum_{i=1}^n\omega_{p_{S_i}\xi_i, p_{S_i}\eta_i}(x_i)+\lambda\omega(1-\pi_S).
$$ 
In particular, 
$$
\norm{\omega_{\restriction f(S)}}\leq\sum_{i=1}^n\Vert p_{S_i}\xi_i\Vert\,\Vert p_{S_i}\eta_i\Vert+\vert\omega(1-\pi_S)\vert.
$$ 
By considering the element of the unit ball of $f(S)$ defined by $$x\coloneqq\sum_{i=1}^ne_{ii}\ot 
\frac{\la p_{S_i}\eta_i,\cdot \ra p_{S_i}\xi_i}
{\Vert p_{S_i}\xi_i\Vert\,\Vert p_{S_i}\eta_i\Vert}
+\frac{\vert\omega(1-\pi_S)\vert}{\omega(1-\pi_S)}(1-\pi_S)$$ where each term with a vanishing denominator is replaced by $0$, we find:
$$
\norm{\omega_{\restriction f(S)}}=\sum_{i=1}^n\Vert p_{S_i}\xi_i\Vert\,\Vert p_{S_i}\eta_i\Vert+\vert\omega(1-\pi_S)\vert.$$ 
Observe that for all $n\in\N$,  the map $S=(S_n)_n\mapsto p_{S_{n}}$ is continuous for the strong topology since for every basic vector $e_j$ we have  
$p_{S_{n}}(e_j)=\mathds 1_{S_n}(j)$.
Similarly, we have that $S\mapsto \pi_S$ is continuous for the strong topology.
We conclude that when $\omega$ is of the form 
\[
\omega=\sum_{i,j=1}^n\omega_{e_i\ot\xi_i,e_j\ot\eta_j},
\]
the map $S\mapsto \norm{\omega_{\restriction f(S)}}$ is continuous. As explained before, by density this finishes the proof.
\end{proof}


\bibliographystyle{alphaurl}
\bibliography{biblio}

\begin{thebibliography}{Rud07}

\bibitem[Bee91]{beerPolishtopologyclosed1991}
Gerald Beer.
\newblock A {{Polish}} topology for the closed subsets of a {{Polish}} space.
\newblock {\em Proceedings of the American Mathematical Society},
  113(4):1123--1133, 1991.
\newblock \href {https://doi.org/10.1090/S0002-9939-1991-1065940-6}
  {\path{doi:10.1090/S0002-9939-1991-1065940-6}}.

\bibitem[Bee93]{beerTopologiesClosedClosed1993}
Gerald~A. Beer.
\newblock {\em Topologies on Closed and Closed Convex Sets}.
\newblock Number 268 in Mathematics and Its Applications. {Kluwer Academic
  Publ}, {Dordrecht}, 1993.

\bibitem[Con00]{conwayCourseOperatorTheory2000}
John~B. Conway.
\newblock {\em A {{Course}} in {{Operator Theory}}}.
\newblock American Mathematical Soc., 2000.

\bibitem[Dyd03]{dydakPartitionsUnity2003}
Jerzy Dydak.
\newblock Partitions of unity.
\newblock {\em Topology Proceedings}, 27(1):125--171, 2003.

\bibitem[Eff65]{effrosBorelspaceNeumann1965}
Edward~G. Effros.
\newblock The {{Borel}} space of von {{Neumann}} algebras on a separable
  {{Hilbert}} space.
\newblock {\em Pacific Journal of Mathematics}, 15(4):1153--1164, 1965.

\bibitem[HW98]{haagerupEffrosMarechalTopologySpace1998}
Uffe Haagerup and Carl Winsl{\o}w.
\newblock The {{Effros-Mar\'echal Topology}} in the {{Space}} of {{Von Neumann
  Algebras}}.
\newblock {\em American Journal of Mathematics}, 120(3):567--617, 1998.

\bibitem[HW00]{haagerupEffrosMarechalTopology2000}
Uffe Haagerup and Carl Winsl{\o}w.
\newblock The {{Effros}}\textendash{{Mar\'echal Topology}} in the {{Space}} of
  von {{Neumann Algebras}}, {{II}}.
\newblock {\em Journal of Functional Analysis}, 171(2):401--431, March 2000.
\newblock \href {https://doi.org/10.1006/jfan.1999.3538}
  {\path{doi:10.1006/jfan.1999.3538}}.

\bibitem[Kec95]{kechrisClassicaldescriptiveset1995}
Alexander~S. Kechris.
\newblock {\em Classical Descriptive Set Theory}, volume 156 of {\em Graduate
  {{Texts}} in {{Mathematics}}}.
\newblock {Springer-Verlag, New York}, 1995.
\newblock \href {https://doi.org/10.1007/978-1-4612-4190-4}
  {\path{doi:10.1007/978-1-4612-4190-4}}.

\bibitem[Kur66]{kuratowskiTopologyVolume1966}
K.~Kuratowski.
\newblock {\em Topology : {{Volume I}}}.
\newblock Academic Press, 1966.
\newblock \href {https://doi.org/10.1016/C2013-0-11022-7}
  {\path{doi:10.1016/C2013-0-11022-7}}.

\bibitem[Li92]{liIntroductionOperatorAlgebras1992}
Bing-Ren Li.
\newblock {\em Introduction to Operator Algebras}.
\newblock World Scientific, Singapore, 1992.

\bibitem[Mar73]{marechalTopologieStructureBorelienne1973}
Odile Mar{\'e}chal.
\newblock Topologie et structure bor\'elienne sur l'ensemble des alg\`ebres de
  von {{Neumann}}.
\newblock {\em C. R. Acad. Sci. Paris Ser. I Math.}, 276:847--850, 1973.

\bibitem[Mic51]{michaelTopologiesSpacesSubsets1951}
Ernest Michael.
\newblock Topologies on spaces of subsets.
\newblock {\em Transactions of the American Mathematical Society},
  71(1):152--182, 1951.
\newblock \href {https://doi.org/10.1090/S0002-9947-1951-0042109-4}
  {\path{doi:10.1090/S0002-9947-1951-0042109-4}}.

\bibitem[Mic56]{michaelContinuousSelections1956}
Ernest Michael.
\newblock Continuous {{Selections}}. {{I}}.
\newblock {\em Annals of Mathematics}, 63(2):361--382, 1956.
\newblock \href {https://doi.org/10.2307/1969615} {\path{doi:10.2307/1969615}}.

\bibitem[Nie73]{nielsenBorelSetsNeumann1973}
Ole~A. Nielsen.
\newblock Borel {{Sets}} of {{Von Neumann Algebras}}.
\newblock {\em American Journal of Mathematics}, 95(1):145--164, 1973.
\newblock \href {https://doi.org/10.2307/2373648} {\path{doi:10.2307/2373648}}.

\bibitem[Ped89]{pedersenAnalysisNow1989}
Gert~K. Pedersen.
\newblock {\em Analysis Now}, volume 118 of {\em Grad. {{Texts Math}}.}
\newblock Springer-Verlag, New York etc., 1989.
\newblock \href {https://doi.org/10.1007/978-1-4612-1007-8}
  {\path{doi:10.1007/978-1-4612-1007-8}}.

\bibitem[Rud07]{rudinFunctionalAnalysis2007}
Walter Rudin.
\newblock {\em {Functional Analysis}}.
\newblock Mcgraw Hill Higher Education, New Delhi, 2007.

\bibitem[Sak57]{sakai1}
Sh{\^o}ichir{\^o} Sakai.
\newblock {On topological properties of $W^*$ algebras}.
\newblock {\em Proceedings of the Japan Academy}, 33(8):439 -- 444, 1957.
\newblock \href {https://doi.org/10.3792/pja/1195524953}
  {\path{doi:10.3792/pja/1195524953}}.

\bibitem[Sak13]{sakaiGeometricAspectsGeneral2013}
Katsuro Sakai.
\newblock {\em Geometric {{Aspects}} of {{General Topology}}}.
\newblock Springer {{Monographs}} in {{Mathematics}}. Springer Japan, 2013.
\newblock \href {https://doi.org/10.1007/978-4-431-54397-8}
  {\path{doi:10.1007/978-4-431-54397-8}}.

\bibitem[Sri98]{srivastavacourseBorelsets1998}
S.~M. Srivastava.
\newblock {\em A Course on {{Borel}} Sets}, volume 180 of {\em Graduate
  {{Texts}} in {{Mathematics}}}.
\newblock Springer-Verlag, New York, 1998.
\newblock \href {https://doi.org/10.1007/978-3-642-85473-6}
  {\path{doi:10.1007/978-3-642-85473-6}}.

\end{thebibliography}

\end{document}